\documentclass[amssymb,amsfonts,refcheck,12pt,verbatim,righttag]{amsart}

\usepackage{graphicx}
\usepackage{graphics,color}
\usepackage{amssymb,mathrsfs}
\usepackage{graphicx,color,tikz,caption,subcaption}
\usepackage[all]{xy}

\usepackage{enumerate}

\usepackage{bm}

\usepackage[colorlinks, linkcolor=blue,citecolor=blue, anchorcolor=blue,urlcolor=blue,pagebackref,hypertexnames=false]{hyperref}

\numberwithin{equation}{section} 
\theoremstyle{plain}

\newtheorem{thm}{Theorem}[section]
\newtheorem{lem}[thm]{Lemma}
\newtheorem{prop}[thm]{Proposition}

\newtheorem{exa}[thm]{Example}
\newtheorem{defn}[thm]{Definition}
\newtheorem{rmk}[thm]{Remark}

\newcommand{\R}{\mathbb{R}}
\newcommand{\C}{\mathbb{C}}
\newcommand{\Q}{\mathbb{Q}}
\newcommand{\N}{\mathbb{N}}
\newcommand{\Z}{\mathbb{Z}}

\newcommand{\A}{\mathcal{A}}

\newcommand{\D}{\mathcal{D}}

\newcommand{\F}{\mathcal{F}}

\usepackage{float}

\begin{document}
\baselineskip 14pt

\title{Matrix representations for some self-similar measures  on $\R^{d}$ }

\author{Yu-Feng Wu}
\address[]{Department of Mathematics\\ The Chinese University of Hong Kong\\ Shatin,  Hong Kong}
\email{yufengwu.wu@gmail.com}



\begin{abstract}
We establish matrix representations for self-similar measures on   $\R^d$
generated by  equicontractive IFSs  satisfying the finite type condition. As an application, we prove  that the $L^q$-spectrum of   every  such  self-similar measure  is differentiable on $(0,\infty)$. This extends an earlier result of Feng (J. Lond. Math. Soc. (2) 68(1):102--118, 2003) to higher dimensions.
\end{abstract}

\maketitle

\section{Introduction}
In this paper, we study self-similar measures on $\R^d$ generated by iterated function systems satisfying the finite type condition. By an iterated function system (IFS) (of similitudes)  on $\R^d$, we mean a finite family of contracting similitudes on $\R^d$.    According to a result of Hutchinson \cite{hutchinson}, given an IFS $\Phi=\{S_i\}_{i=1}^m$ on $\R^d$, there is a unique non-empty compact set $K\subset \R^d$ satisfying  $K=\bigcup_{i=1}^m S_i(K)$,  which is called the {\em self-similar set} generated by $\Phi$. 
Moreover, given a probability vector ${\bm p}=(p_1,\ldots,p_m)$, i.e. each $p_i>0$ and $\sum_{i=1}^mp_i=1$, there exists a unique Borel probability measure $\mu$ supported on $K$ such that 
\begin{equation}\label{eq: similarity identity}
\mu=\sum_{i=1}^mp_i\mu\circ S_i^{-1}.
\end{equation}
We call  $\mu$ the {\em self-similar measure} generated by $\Phi$ and  ${\bm p}$. 

Given an IFS $\Phi=\{S_i\}_{i=1}^m$ on $\R^d$,  let  $\Sigma=\{1, \ldots, m\}$ be the alphabet associated to $\Phi$.  For each $n\in\N$, let $\Sigma_n=\{i_1\ldots i_n: i_k\in\Sigma,1\leq k\leq n\}$.
Moreover, set $\Sigma_0=\{\varepsilon\}$, where $\varepsilon$ denotes the empty word. Let $\Sigma_*=\bigcup_{n=0}^{\infty}\Sigma_n$ and let  $\Sigma^{\N}$ be the collection of infinite words over $\Sigma$.  For  $I=i_1\ldots i_n\in \Sigma_*$,  write $S_I=S_{i_1}\circ\cdots\circ S_{i_n}$. In particular, set $S_{\varepsilon}=id$, the identity map on $\R^d$. For a similitude $S$ on $\R^d$, we let $\rho_S>0$ be the similarity ratio of $S$.   We say that $\Phi$  is   {\em equicontractive} if   $\rho_{S_1}=\cdots=\rho_{S_m}$.

This paper is motivated by the study of equicontractive self-similar measures on $\R$  by Feng \cite{feng1}. More precisely, let $\mu$ be a self-similar measure on $\R$ generated by an  IFS $\Phi=\{S_i\}_{i=1}^m$ on $\R$ of the form 
\begin{equation}\label{eq: Feng's model}
S_i(x)=\lambda x+b_i, \quad  i=1, \ldots, m,
\end{equation}
where $\lambda\in(0,1)$ and $b_1,\ldots, b_m\in \R$. In \cite{feng1}, Feng investigated the case when $\Phi$ satisfies the so-called   finite type condition, i.e., there exists a finite set $\Gamma_0$ such that for any $n\geq 1$ and any  $I, J\in\Sigma_n$, 
\begin{equation}\label{eq: Feng's FTC}
\text{either }\quad \lambda^{-n}|S_I(0)-S_J(0)|\geq c \quad \text{ or } \quad \lambda^{-n}|S_I(0)-S_J(0)|\in\Gamma_0,
\end{equation}
where $c=(1-\lambda)^{-1}(\max_{1\leq i\leq m}b_i-\min_{1\leq i\leq m}b_i)$. Feng established the matrix representations for $\mu$ on the so-called basic net intervals and proved that    the $L^q$-spectrum  of $\mu$ (see \eqref{eqdefLq} for the definition)  is differentiable on $(0,\infty)$. The result of  matrix representations for $\mu$ is also used to give a checkable criterion for the absolute continuity of $\mu$ (cf. \cite[Theorem 6.2]{FLS20}), and   it is also  applied to study the topological structure of the set of local dimensions of $\mu$ in a series of papers \cite{HHM16,HHN18,HHSi18}.

The purpose of this paper is to extend the results of  \cite{feng1} to higher dimensions. We will consider self-similar measures on $\R^d$ generated by equicontractive IFSs  satisfying the following version of  the finite type condition. 

\begin{defn}\label{deFTC}
Let $\Phi=\{S_i\}_{i=1}^m$ be an equicontractive IFS on $\R^d$ which generates a self-similar set  $K$. 
 We say that $\Phi$ satisfies the finite type condition (FTC) if there exists a finite set $\Gamma$ such that for any $n\geq 1$ and $I, J\in\Sigma_n$,
\begin{equation}\label{SFTC}
\text{either }\quad S_I(K)\cap S_J(K)=\emptyset \quad  \text{ or }\quad S_I^{-1}\circ S_J\in\Gamma.
\end{equation}
\end{defn}

The concept of the FTC was first introduced by Ngai and Wang \cite{NgaiWang} in a more general setting.  It is easily seen that  for an IFS on $\R$ of the form \eqref{eq: Feng's model}, the conditions \eqref{eq: Feng's FTC} and \eqref{SFTC} are  equivalent. 

To state our   main results of this paper, we first introduce some notation. First we define the (canonical) Borel partitions of a self-similar set generated by an equicontractive IFS.

Let $K\subset\R^d$ be the  self-similar set  generated by  an equicontractive IFS $\Phi=\{S_i\}_{i=1}^m$. Let $\mathcal{S}$ denote the set of  similitudes on $\R^d$ and $2^{\mathcal{S}}$ be the collection of all  the subsets of $\mathcal{S}$. 
For each $n\geq 0$, we define a mapping $\Lambda_n: K\to 2^{\mathcal{S}}$ by
\begin{equation}\label{the map Lambda_n0}
\Lambda_n(x)=\{S_I: I\in \Sigma_n \text{ with }x\in S_I(K)\} \quad \ \text{ for }   x\in K.
\end{equation}
Let $\Lambda_n(K)$ be the image of $K$ under $\Lambda_n$, i.e. $\Lambda_n(K)=\{\Lambda_n(x):x\in K\}$. 
Then define
\begin{equation}\label{definexin}
\xi_n=\left\{\Lambda_n^{-1}(\mathcal{U}):\mathcal{U}\in\Lambda_n(K)\right\}.
\end{equation}
It is  easy to see that $\xi_n$ is a finite  partition of $K$ whose elements are all Borel sets. We call $\xi_n$ the $n$-th (canonical) {\em Borel partition} of $K$. 

Also we  need the notion of $L^q$-spectrum of measures. 
Let $\nu$ be   a finite Borel measure  on $\mathbb{R}^d$ with compact support. For  $q\in\R$, the {\em $L^q$-spectrum} of $\nu$ is defined by 
\begin{equation}\label{eqdefLq}
\tau(q)=\tau(\nu, q)=\liminf_{\delta\to0}\frac{\log\left({\sup\sum_i\nu(B(x_i,\delta))^q}\right)}{\log{\delta}},
\end{equation}
where the supremum is taken over all families of disjoint closed balls $B(x_i, \delta)$ of radius $\delta$ and  centres $x_i\in\text{supp}\nu$.

The main results  of this paper are the following two results, which extend \cite[Theorem 1.1]{feng1} from $\R$ to $\R^d$.

\begin{thm}\label{main thm:matrix repn}
Let $\Phi=\{S_i\}_{i=1}^m$ be an equicontractive IFS on $\R^d$ and $K$ be the self-similar set generated by $\Phi$. Suppose that  $\Phi$  satisfies  the FTC.  Let $\mu$ be
the self-similar  measure generated by $\Phi$ and a  probability vector $(p_1,\ldots,p_m)$. 
Then there exist  $s, N\in\N$, $N\times N$ non-negative  matrices $M_1,\ldots, M_s$,  and $N$-dimensional  positive row vectors $\mathbf{w}_1,\ldots, \mathbf{w}_s$ such that for any $n\in\N$ and $\Delta\in \xi_n$ with $\mu(\Delta)>0$, we have 
\[\mu(\Delta)=\bm{e}_1M_{i_1}\cdots M_{i_n}\mathbf{w}_{i_n}^T,\]
where $\eta_{1}\eta_{i_1}\ldots\eta_{i_n}$ is the symbolic expression of $\Delta$ (see Section \ref{section: symbolic expression} for the definition),   $\bm{e}_1=(1,0,\ldots,0)\in\R^N$ and   $\bm{a}^T$ denotes the transpose of $\bm{a}$.
\end{thm}

\begin{thm}\label{main thm}Under the assumptions of Theorem \ref{main thm:matrix repn}, the $L^q$-spectrum $\tau(q)$ of $\mu$ is differentiable on $(0,\infty)$.
\end{thm}

We present two examples to which Theorems \ref{main thm:matrix repn}-\ref{main thm}  apply directly. Recall that an algebraic integer  $\beta>1$ is called a {\em Pisot number} if all its  
Galois conjugates have modulus less than $1$. Similarly, we call  an algebraic integer $\zeta\in\C\setminus \R$ a {\em complex Pisot number} if $|\zeta|>1$ and all its Galois conjugates, except $\bar{\zeta}$, have modulus less than $1$. 

\begin{exa}\label{exam: Pisot without rotation}\cite[Theorem 2.5]{NgaiWang}\label{exa1} Let $\Phi=\{\rho x+a_i\}_{i=1}^m$ be an IFS on $\R^d$, where $\rho^{-1}$ is a Pisot number and $a_i\in\Q(\rho)^d$ for $1\leq i\leq m$. Then $\Phi$ satisfies the FTC.
\end{exa}

\begin{exa}\label{rotation with complex Pisot number}Let $\Phi=\{\rho x+a_i\}_{i=1}^m$ be an IFS on $\C$ in the complex form, where  $\rho^{-1}$ is a complex Pisot number and $a_i\in\Q(\rho)$ for $1\leq i\leq m$. Then $\Phi$ satisfies the FTC.
\end{exa}

The proof of  Example \ref{rotation with complex Pisot number}  is similar to that of Example \ref{exam: Pisot without rotation}, which we omit.

Theorems \ref{main thm:matrix repn}-\ref{main thm} can also be applied to self-similar measures on the so-called golden gaskets studied in \cite{BMS04}. Indeed, more generally, let $\Phi=\{S_i(x)=\lambda x+(1-\lambda)a_i\}_{i=1}^3$ be an IFS on $\R^2$, where $\lambda^{-1}$ is a Pisot number, and $a_1,a_2,a_3\in \R^2$ are non-collinear points.\footnote{The self-similar set generated by $\Phi$ is called a golden gasket in \cite{BMS04} if $\lambda^{-1}$ is a multinacci number and  $a_1, a_2, a_3$ are  vertices of an equilateral triangle.}
  Then there exists an invertible affine transformation $f$ on $\R^2$ such that $$\{f\circ S_i\circ f^{-1}\}_{i=1}^3=\{\lambda x, \lambda x+(1,0), \lambda x+(0,1)\},$$ 
 see \cite[Section 8, Remark (1)]{BMS04}.  Then by Example \ref{exa1}  the IFS $\{f\circ S_i\circ f^{-1}\}_{i=1}^3$ satisfies the FTC. It then follows easily from the definition that $\Phi$ itself satisfies the FTC.

The $L^q$-spectrum  is one of the basic ingredients in multifractal analysis. There is a  well-known heuristic relation between the $L^q$-spectrum and the {\em dimension spectrum} of a measure called the {\em multifractal formalism}; see \cite{falconer1,pesin} for the definitions and detailed properties of these notions.  There have been a lot of studies on  the multifractal formalism for self-conformal (including self-similar) and self-affine measures, see e.g. \cite{cm,olsen,LauNgai,feng2,fenglau2,barralfeng,BF20}. 

For a self-similar measure  generated by an IFS satisfying the {\em open set condition} (OSC) \cite{hutchinson}, it is well-known that its $L^q$-spectrum $\tau(q)$ is given by a precise  formula and is analytic on $\mathbb{R}$ (cf. \cite{cm}). For a general self-similar measure, it is known \cite{ps} that the limit in \eqref{eqdefLq} always exists for $q\geq0$. However, without  the OSC, it is generally difficult to obtain a formula for $\tau(q)$; see \cite{LN98,LN99,LN00,feng1,feng05,fenglau2,HHS19} for related works. 
For some self-similar measures,  $\tau$ may not be differentiable at some points. For instance,  Feng \cite{feng05} showed that for the Bernoulli convolution associated with the golden ratio, $\tau$ is differentiable on $\mathbb{R}$ except at one point $q_0<0$. Moreover,   Barral and Feng \cite{barralfeng} showed that for any $q\in(1,2)$, there exists a self-similar measure on $\mathbb{R}$ for which $\tau$ is not differentiable at $q$. According to a recent result of Shmerkin \cite[Theorem 6.6]{Shmerkin19}, the above result also extends to $q\in (1,\infty)$. It remains an interesting question to determine for what classes of self-similar measures, $\tau$ is differentiable on $(0,\infty)$. Only a few results have been obtained on this question. As we have mentioned, Feng  \cite{feng1} proved the differentiability of $\tau$  on $(0,\infty)$ for the self-similar measures on $\R$ generated by the  IFS of the form  \eqref{eq: Feng's model} satisfying \eqref{eq: Feng's FTC}. Recently,  in \cite{dn,NgaiXie2019} the same conclusion was proved  for certain self-similar measures on $\R^d$ that satisfy the {\em generalized finite type condition} (GFTC) and that are of {\em essentially finite type} (EFT), respectively. Below we make some comments on how  the assumptions in \cite{dn,NgaiXie2019} are related to that  in  this paper. 

 It is known that IFSs satisfying our definition of the FTC satisfy the GFTC; see \cite[Theorem 1.1]{denglaonagi13}. However, the key of the approach used in \cite{dn} is to  construct an infinite graph-directed IFS satisfying the OSC from a finite IFS, which relies  heavily on the particular structure of the corresponding self-similar set; see \cite[Subsection 4.2]{dn}. Thus (and as mentioned in \cite[Section 9]{dn}) it is not clear whether the method in \cite{dn} can be extended to cover other IFSs satisfying the GFTC, including those studied by Feng \cite{feng1}.  As for the condition  EFT used in \cite{NgaiXie2019}, it is known that for the special case of Example \ref{exa1} when $\mu$ is the Bernoulli convolution associated with the golden ratio,  the EFT is satisfied (cf. \cite[Example 3.2]{NgaiTangXie18}). As shown in the proof of \cite[Example 3.2]{NgaiTangXie18}, this relies on the fact that  $\mu$ satisfies the so called second-order identities (cf. \cite[Equation 1.5]{LN98}). However, to the best of the author's knowledge, for Bernoulli convolutions associated with Pisot numbers other than the golden ratio, no second-order identities have been proved (this situation is also mentioned in the end of the second page of \cite{NgaiTangXie18}), and the EFT has not been verified. Thus it is not clear whether the main result of \cite{NgaiXie2019} (i.e. \cite[Theorem 1.1]{NgaiXie2019}) applies to this case, which is also stated as an unsolved problem in \cite[Section 6]{NgaiXie2019}.

In the study of self-similar sets and measures,  extending results from $\R$ to $\R^d$ can often be difficult when the OSC fails. One reason is that unlike the case in $\R$, in $\R^d$ when $d\geq 2$ the orthogonal matrices in the  linear parts of the similitudes in the IFS may cause obstacles; see for instance \cite{hochman}. In the following, we give a description of our method used in this paper in extending Feng's work \cite{feng1} from $\R$ to $\R^d$.

Our general strategy to prove the main results  is analogous to that of \cite[Theorem 1.1]{feng1}, but several new ideas play key roles in our situation. First, for a self-similar set $K$ on $\R^d$ generated by an equicontractive IFS  satisfying the FTC, we introduce  the notion of   Borel partitions of $K$, which is analogous to the notion of basic net intervals  in \cite{feng1}. Second,   we introduce a linear order  on the set  of the compositions of the IFS and use it to define  the  characteristic vectors and the symbolic expressions for the elements of the Borel partitions (cf. Section \ref{section: symbolic expression}). This linear order enables us to  avoid possible difficulties caused by the rotations in the IFS.   Third,  in the  proof of  Theorem \ref{main thm},  as in \cite{feng1} our strategy is to  connect  $\tau(q)$ with the pressure function  $P(q)$ for a certain family of squared matrices (cf. Theorem \ref{thm: t(q) and P(q)}); see Subsection \ref{section: proof of main thm} for the definition of $P(q)$. Then Theorem \ref{main thm} follows from a result of Feng and Lau \cite{fenglau1} which states that  $P(q)$ is differentiable on $(0, \infty)$ under the condition that the sum of these  matrices is  irreducible.  A key difference is that, to verify the irreducibility condition, we make use of the Borel density lemma  (cf.   Lemma \ref{lem: mu(E)=1}), which is not necessary in the one dimensional case studied in \cite{feng1} due to the fact that the interior of each basic net interval intersects $K$. 

Our method can be slightly  extended to the  more general case that the IFS is commensurable and satisfies a more general form of  the FTC; see Theorem  \ref{mainthm:com case}.  Since the proof of Theorem \ref{mainthm:com case} uses essentially the same ideas of that of Theorems \ref{main thm:matrix repn}-\ref{main thm}, we will first prove  Theorems \ref{main thm:matrix repn}-\ref{main thm}, and then  we point out in Section \ref{sec:comensurable case} the  modifications needed in the proofs of  Theorems \ref{main thm:matrix repn}-\ref{main thm}  to prove Theorem \ref{mainthm:com case}.

As we have mentioned, for a self-similar measure $\mu$ on $\R$ satisfying the FTC, Feng's result in \cite{feng1} has been applied in \cite{FLS20} to give a checkable criterion for the absolute continuity of $\mu$, and in \cite{HHM16,HHN18,HHSi18} to study the topological structure of the set of local dimensions of $\mu$. It would be interesting to see if our method in this paper can be applied to these topics, which however is beyond the scope of the present paper. 

The rest of this paper is organized as follows. Throughout Sections \ref{section: symbolic expression}-\ref{section: squared matrices}, we let $\Phi=\{S_i\}_{i=1}^m$ be an equicontractive IFS on $\R^d$  satisfying the FTC,  $K$ be the self-similar set generated by $\Phi$  and $\mu$ be the self-similar measure generated by  $\Phi$ and  a probability vector  $(p_1,\ldots, p_m)$.   In Section \ref{section: symbolic expression}, we define  the characteristic vector and symbolic expression for each element of $\xi_n$ ($n\geq 0$). In Section \ref{section: distribution of mu}, we prove Theorem \ref{main thm:matrix repn}. Section \ref{section: squared matrices} is devoted to the proof of Theorem \ref{main thm}. In Section \ref{sec:comensurable case}, we show how to  modify the developments in Sections  \ref{section: symbolic expression}-\ref{section: squared matrices}   to prove Theorem \ref{mainthm:com case}.

\section{The characteristic vectors and symbolic expressions of  \texorpdfstring{$\Delta\in\xi_n$}{Delta}}\label{section: symbolic expression}

For $n\geq 0$, let  the mapping $\Lambda_n$, and the $n$-th Borel partition $\xi_n$ of $K$  be  defined as in \eqref{the map Lambda_n0} and \eqref{definexin}, respectively. It is clear that $\Lambda_n$ takes a constant value on each $\Delta\in \xi_n$, which we denote by $\Lambda_n(\Delta)$. Define the ``neighbor" of $\Lambda_n(\Delta)$ by
\begin{equation*}\label{DefNnD}
N_n(\Delta)=\left\{S_I: I\in\Sigma_n \text{ with } S_I(K)\cap \left(\bigcap_{f\in\Lambda_n(\Delta)}f(K)\right)\neq \emptyset \right\}.
\end{equation*}
Clearly, $\Lambda_n(\Delta)\subseteq N_n(\Delta)$. Moreover,  $\Delta$ is determined by $\Lambda_n(\Delta)$ and $N_n(\Delta)$, since 
\begin{align}
\Delta&=\left(\bigcap_{f\in\Lambda_n(\Delta)}f(K)\right)\backslash\left(\bigcup_{I\in \Sigma_n:\ S_I\notin\Lambda_n(\Delta)}S_I(K)\right)\label{eq: defining property} \\
&=\left(\bigcap_{f\in\Lambda_n(\Delta)}f(K)\right)\backslash\left(\bigcup_{g\in N_n(\Delta)\backslash\Lambda_n(\Delta)}g(K)\right). \label{eqDeltaeqleft}
\end{align}

In the following, we present some basic properties of  $\xi_n$, $\Lambda_n(\Delta)$ and $N_n(\Delta)$. First, notice that the elements of $\xi_n$ are precisely the atoms of the algebra generated by $\{S_I(K): I\in \Sigma_n\}$ (see e.g. \cite[p.86, p.115]{Dudley} for the definitions of an algebra of sets and its atoms). Hence we have the following two lemmas.  

\begin{lem}\label{lemcharac}
Let $n\geq 0$, $\mathcal{U}\subseteq\{S_I: I\in\Sigma_n\}$ and $$\mathcal{V}=\left\{S_J: J\in \Sigma_n \text{ with } S_J(K)\cap\left(\bigcap_{f\in\mathcal{U}} f(K)\right)\neq \emptyset\right\}.$$ Set $\Delta=\left(\bigcap_{f\in\mathcal{U}}f(K)\right)\backslash\left(\bigcup_{g\in\mathcal{V}\backslash\mathcal{U}}g(K)\right)$. Then $\Delta\in\xi_n$ if and only if $\Delta\neq\emptyset$. In the case when $\Delta\neq\emptyset$, we have $\Lambda_n(\Delta)=\mathcal{U}$ and $N_n(\Delta)=\mathcal{V}$.
\end{lem}

\begin{lem}\label{lem2-3}
Let $n\geq 0$ and $\mathcal{U},\mathcal{V}\subseteq\{S_I:I\in\Sigma_n\}$. If $\left(\bigcap_{f\in\mathcal{U}}f(K)\right)\backslash\left(\bigcup_{g\in\mathcal{V}}g(K)\right)$ is non-empty, then it 
is a union of some elements of $\xi_n$.
\end{lem}

The following lemma shows that $\{\xi_n\}_{n=0}^{\infty}$ has a net structure. 

\begin{lem}\label{lem: xi(n+1) refines xi(n)}
{\rm (i)} For any $n\geq 0$, $\xi_{n+1}$ refines $\xi_n$. That is,  each element of $\xi_{n+1}$ is a subset of an  element of $\xi_n$.

{\rm (ii)} Moreover, given $\Delta\in\xi_{n+1}$ let $\widehat{\Delta}\in\xi_{n}$ such that  $\Delta\subseteq\widehat{\Delta}$, then 
\[\Lambda_n(\widehat{\Delta})=\left\{S_I: I\in\Sigma_n, \exists S\in\Phi \text{ such that }S_I\circ S\in \Lambda_{n+1}(\Delta)\right\}. \]
\end{lem}

\begin{proof} Let $n\geq 0$ and $\Delta\in\xi_{n+1}$.  Set
\begin{equation}\label{defmathcalU}
\mathcal{U}=\left\{S_I: I\in\Sigma_n, \exists S\in\Phi \text{ such that }S_I\circ S\in \Lambda_{n+1}(\Delta)\right\},
\end{equation}
$$\mathcal{V}=\left\{S_J: J\in \Sigma_n \text{ with } S_J(K)\cap\left(\bigcap_{f\in\mathcal{U}} f(K)\right)\neq \emptyset\right\}.$$
Define
\begin{equation}\label{eq: net property0}
\widehat{\Delta}=\left(\bigcap_{f\in\mathcal{U}}f(K)\right)\backslash\left(\bigcup_{g\in\mathcal{V}\backslash\mathcal{U}}g(K)\right).
\end{equation}
Below we show that  $\widehat{\Delta}\in\xi_n$, $\Delta\subseteq\widehat{\Delta}$ and $\Lambda_{n}(\widehat{\Delta})=\mathcal{U}$, which will prove the lemma. 

For every $h\in \Lambda_{n+1}(\Delta)$, we have $h=S_I\circ S$ for some $I\in \Sigma_n$ and $S\in \Phi$, in which case $h(K)=S_I\circ S(K)\subseteq S_I(K)$. By this fact and \eqref{defmathcalU} we see that
\begin{equation}\label{eq: net structure1}\bigcap_{h\in\Lambda_{n+1}(\Delta)}h(K)\subseteq\bigcap_{f\in \mathcal{U}}f(K).
\end{equation}
On the other hand,  \eqref{defmathcalU} also implies that $g\circ S\notin \Lambda_{n+1}(\Delta)$ for all   $g\in\mathcal{V}\setminus\mathcal{U}$ and $S\in\Phi$. Since $g(K)=\bigcup_{i=1}^mg\circ S_i(K)$ for every $g\in \mathcal{V}\setminus\mathcal{U}$, it follows that 
\begin{equation}\label{eq: net structure2}
\bigcup_{I\in \Sigma_{n+1}:\ S_I\notin \Lambda_{n+1}(\Delta)}S_I(K)\supseteq \bigcup_{g\in \mathcal{V}\setminus\mathcal{U}}g(K).
\end{equation}
Recall that 
\begin{equation}\label{eq:Delta n+1}
\Delta=\left(\bigcap_{h\in \Lambda_{n+1}(\Delta)}h(K)\right)\backslash \left(\bigcup_{I\in \Sigma_{n+1}:\ S_I\notin \Lambda_{n+1}(\Delta)}S_I(K)\right).
\end{equation} 
Now by \eqref{eq: net property0}-\eqref{eq:Delta n+1} we see  that  $\Delta\subseteq \widehat{\Delta}$. Hence $\widehat{\Delta}\neq \emptyset$ as  $\Delta\neq\emptyset$. It then follows from Lemma  \ref{lemcharac} that   $\widehat{\Delta}\in \xi_n$ and $\Lambda_n(\widehat{\Delta})=\mathcal{U}$. This proves the lemma. \qed 
\end{proof}

  Recall that $\mathcal{S}$ is the set of similitudes on $\R^d$. Here and afterwards, for $g\in\mathcal{S}$, $A\subset \mathcal{S}$ and a vector $V=(f_1,\ldots,f_n)$ with all $f_i\in\mathcal{S}$, we write $g\circ A=\{g\circ f: f\in A\}$ and $g\circ V=(g\circ f_1,\ldots, g\circ f_n)$.

  Let $\#A$ denote the cardinality of a set $A$. The following result is a direct consequence of the FTC.
\begin{lem}\label{lem: FTC}
For any $n\in\N$ and $\Delta\in \xi_n$,  $\#\Lambda_n(\Delta)\leq \#N_n(\Delta)\leq \#\Gamma$, where $\Gamma$ is given as in Definition \ref{deFTC}.
\end{lem}

\begin{proof}
Let $n\in\N$ and $\Delta\in \xi_n$. The first inequality is clear as  $\Lambda_n(\Delta)\subseteq N_n(\Delta)$. To see the second inequality,  fix  $f\in\Lambda_n(\Delta)$. Then by \eqref{SFTC} and   the definition of $N_n(\Delta)$, we see that $f^{-1}\circ N_n(\Delta)\subseteq\Gamma.$
Hence  $\#N_n(\Delta)\leq \#\Gamma$. \qed
\end{proof}

In the rest of this section, we define for each $\Delta\in\xi_n$ ($n\geq 0$) its characteristic vector and symbolic expression. To this end, we first   introduce a  linear order on $\Phi_*:=\{S_I: I\in\Sigma_*\}$, which enables us to rewrite $\Lambda_n(\Delta)$ and $N_n(\Delta)$ as ordered vectors and  is important for our further analysis (cf. Lemma \ref{lemporder}). 

Let $\leqslant_{\rm lex}$ be the lexicographic order on $\Sigma_*$. That is, for   $I=i_1\ldots i_k,$ $J=j_1\ldots j_{\ell}\in \Sigma_*$,  $I\leqslant_{\rm lex} J$ if and only if  either $I$ is a prefix of $J$, or there exists  $1\leq s\leq \min\{k,\ell\}$ such that $i_1=j_1, \ldots,  i_{s-1}=j_{s-1}$  and  $i_s<j_s$. 
Write $I<_{\rm lex} J$ if $I\leqslant_{\rm lex} J$ and $I\neq J$.  It is easy to check that $\leqslant_{\rm lex}$  is  a linear order on $\Sigma_*$. Moreover, $\leqslant_{\rm lex}$ satisfies  the following property, which  is obvious from the definition. 

\begin{lem}\label{lembpoforder}
Let $n\geq 0$ and $I=i_1\ldots i_{n+1}, J=j_1\ldots j_{n+1}\in \Sigma_{n+1}$. Then $I<_{\rm lex} J$ if and only if either $i_1\ldots i_n <_{\rm lex} j_1\ldots j_n$, or $i_1\ldots i_n=j_1\ldots j_n$ and $i_{n+1}<j_{n+1}$. 
\end{lem}

Using  $\leqslant_{\rm lex} $ on $\Sigma_*$ we define a linear order $\preccurlyeq $ on $\Phi_*$ as follows: Let $\omega: \Phi_*\to \Sigma_*$ be the mapping which assigns each $f\in \Phi_*$ the minimal  element of $\{I\in \Sigma_*: S_I=f\}$ under $\leqslant_{\rm lex} $. Notice that $\omega$ is well-defined since $\leqslant_{\rm lex} $ is linear.   For $f, g\in\Phi_*$,  write $f\preccurlyeq g$ ($f\prec g$)  if $\omega(f)\leqslant_{\rm lex}  \omega(g)$ ($\omega(f)<_{\rm lex}\omega(g)$, respectively). It is clear that  $\preccurlyeq$ is a linear order on $\Phi_*$.

From now on, for $n\geq 0$ and  $\Delta\in\xi_n$ with $\Lambda_n(\Delta)=\{f_i\}_{i=1}^k$ and  $N_n(\Delta)=\{g_j\}_{j=1}^{\ell}$, without loss of generality we assume  that the elements of  $\Lambda_n(\Delta)$ and $N_n(\Delta)$ are ranked increasingly in the order $\prec$, i.e., $f_1\prec\cdots\prec f_k$ and $g_1\prec \cdots\prec g_{\ell}$. Then without causing confusion we will view   $\Lambda_n(\Delta)$ and $N_n(\Delta)$ as ordered vectors
\[\Lambda_n(\Delta)=(f_1,\ldots, f_k), \quad N_n(\Delta)=(g_1,\ldots, g_{\ell}).\]
We define  two vectors $V_n(\Delta)$ and $U_n(\Delta)$ by
\begin{equation*}\label{defVU}
V_n(\Delta)=(\varphi_1,\ldots, \varphi_k), \quad U_n(\Delta)=(\psi_1, \ldots, \psi_{\ell}),
\end{equation*}
where $\varphi_i=f_1^{-1}\circ f_i$ for $1\leq i\leq k$ and $\psi_j=f_1^{-1}\circ g_j$ for $1\leq j\leq \ell$.  By Lemma \ref{lem: FTC} and its proof, we see that $k\leq \ell\leq \#\Gamma$, and  all the entries of $V_n(\Delta)$ and $U_n(\Delta)$ are contained in $\Gamma$.  Hence  $\{(V_n(\Delta), U_n(\Delta)): \Delta\in \xi_n, n\geq 0\}$ is a finite set. 

Next we define for $\Delta$ a similitude $r_n(\Delta)$ on $\R^d$. Let $r_0(K):=id$, the identity map on $\R^d$. If $n\geq 1$, let $\widehat{\Delta}$ denote the unique element of $\xi_{n-1}$ which contains $\Delta$. Assume that $\Lambda_{n-1}(\widehat{\Delta})=(h_1, \ldots, h_{k'})$. Then we define
$r_n(\Delta)=h_1^{-1}\circ f_1$. By  Lemma \ref{lem: xi(n+1) refines xi(n)}(ii),   there exist $S\in \Phi$ and $j\in\{1,\ldots,k\}$ so that $h_1\circ S=f_j$. Since $f_j(K)\cap f_1(K)\neq\emptyset$, the FTC \eqref{SFTC} implies that $f_j^{-1}\circ f_1\in \Gamma$. Hence we have
\[r_n(\Delta)=h_1^{-1}\circ f_1=S\circ f_j^{-1}\circ f_1\in \left\{S_i\circ f: 1\leq i\leq m, f\in \Gamma\right\}.\]
As a consequence, we see that the set $\{r_n(\Delta): \Delta\in \xi_n,  n\geq 0\}$ is  finite. 

Finally, we define a triple
$$\mathcal{C}_n(\Delta)=(V_n(\Delta), U_n(\Delta), r_n(\Delta)),$$
and call it the {\em characteristic vector} of $\Delta$. Let 
$$\Omega=\{\mathcal{C}_n(\Delta): \Delta\in\xi_n, n\geq 0\}.$$
Then by  the above argument,  $\Omega$ is a finite set.

For $n\geq 0$ and $\Delta\in\xi_n$, 
since $\Delta$ is determined by $\Lambda_n(\Delta)$ and $N_n(\Delta)$ (cf. \eqref{eqDeltaeqleft}), $V_n(\Delta)$ and $U_n(\Delta)$ are used to record the shape of $\Delta$. The reason to introduce the term $r_n(\Delta)$ in $\mathcal{C}_n(\Delta)$ is to guarantee that $\mathcal{C}_{n+1}(\Delta_1)\neq \mathcal{C}_{n+1}(\Delta_2)$ whenever $\Delta_1, \Delta_2\in\xi_{n+1}$ are contained in $\Delta$ and $\Delta_1\neq \Delta_2$, as  shown in the following lemma.

\begin{lem}\label{lemDisch}
Given $n\geq 0$ and $\Delta\in\xi_n$ let $\Delta_1, \Delta_2$ be two distinct  elements of $\xi_{n+1}$ that are contained in $\Delta$. Then $\mathcal{C}_{n+1}(\Delta_1)\neq \mathcal{C}_{n+1}(\Delta_2)$.
\end{lem}

\begin{proof}
Since $\Delta_1\neq\Delta_2$ and they are determined by  $\Lambda_{n+1}(\Delta_1)$ and $\Lambda_{n+1}(\Delta_2)$  respectively (cf. \eqref{eq: defining property}), we have $\Lambda_{n+1}(\Delta_1)\neq \Lambda_{n+1}(\Delta_2)$. Let $f, g, h$ be the first entries of $\Lambda_{n}(\Delta)$, $\Lambda_{n+1}(\Delta_1)$ and $\Lambda_{n+1}(\Delta_2)$, respectively.   If $g\neq h$, then we have $r_{n+1}(\Delta_1)\neq r_{n+1}(\Delta_2)$, since  $r_{n+1}(\Delta_1)=f^{-1}\circ g$ and $r_{n+1}(\Delta_2)=f^{-1}\circ h$. If $g=h$, then since $V_{n+1}(\Delta_1)=g^{-1}\circ\Lambda_{n+1}(\Delta_1)$, $V_{n+1}(\Delta_2)=h^{-1}\circ\Lambda_{n+1}(\Delta_2)$ and $\Lambda_{n+1}(\Delta_1)\neq \Lambda_{n+1}(\Delta_2)$, we see that $V_{n+1}(\Delta_1)\neq V_{n+1}(\Delta_2)$.  Hence we have shown that either $r_{n+1}(\Delta_1)\neq r_{n+1}(\Delta_2)$ or $V_{n+1}(\Delta_1)\neq V_{n+1}(\Delta_2)$. This  implies that  $\mathcal{C}_{n+1}(\Delta_1)\neq \mathcal{C}_{n+1}(\Delta_2)$, as desired. \qed
\end{proof}

Now we proceed to introduce the symbolic expression for each element in $\{\Delta\in\xi_n: n\in\N\}$. For this purpose,  we need establish the following result.  

\begin{lem}\label{lem: char vector}
Let $k, \ell\in \N$, $\Delta_1\in\xi_{k}$ and  $\Delta_2\in\xi_{\ell}$. 
If  $\mathcal{C}_{k}(\Delta_1)=\mathcal{C}_{\ell}(\Delta_2)$, then  
$$\left\{\mathcal{C}_{k+1}(\Delta): \Delta\in \xi_{k+1}, \Delta\subseteq \Delta_1\right\}=\left\{\mathcal{C}_{\ell+1}(\Delta):\Delta\in \xi_{\ell+1}, \Delta\subseteq \Delta_2\right\}.$$
\end{lem}

To prove Lemma \ref{lem: char vector}, we first give  several lemmas. 

In this paper, we always assume that the similitudes in the IFS $\Phi$ are distinct, which is clearly a natural assumption.
 Given  $I=i_1\ldots i_k\in\Sigma_k$ and  $J=j_1 \ldots j_{\ell}\in\Sigma_{\ell}$,  write $IJ=i_1\ldots i_kj_1\ldots j_{\ell}\in\Sigma_{k+\ell}$ for the concatenation of $I$ and $J$. We will use the following simple property of the mapping $\omega:\Phi_*\to\Sigma_*$ that we defined after Lemma \ref{lembpoforder}.

\begin{lem}\label{concatenation lem}
Let $n\geq 0$, $\Delta\in\xi_n$, $\Delta_1\in \xi_{n+1}$ with $\Delta_1\subseteq \Delta$ and assume that $\Lambda_n(\Delta)=(f_1,\ldots, f_k)$. Given $h\in\Lambda_{n+1}(\Delta_1)$ let $i$ be the smallest integer in $\{1,\ldots, k\}$ so that $h=f_i\circ S_j$ for some  $j\in \{1,\ldots, m\}$. Then $\omega(h)= \omega(f_i)j.$
\end{lem}
\begin{proof}
Let $h\in \Lambda_{n+1}(\Delta_1)$. First, from Lemma \ref{lem: xi(n+1) refines xi(n)}(ii) we see that  the above $i,j$  exist. Clearly, $\omega(f_i)\in\Sigma_n$ and $\omega(h)\in\Sigma_{n+1}$. Moreover, since $h=f_i\circ S_j=S_{\omega(f_i)j}$, we have  $\omega(h)\leqslant_{\rm lex} \omega(f_i)j$. Below we prove that   $\omega(h)=\omega(f_i)j$ by  contradiction. 

Suppose on the contrary that $\omega(h)<_{\rm lex}\omega(f_i)j$.
Write $\omega(h)=i_1\ldots i_{n+1}$.  Then by Lemma \ref{lembpoforder}, either $i_1\ldots i_n<_{\rm lex}\omega(f_i)$, or $i_1\ldots i_n=\omega(f_i)$ and $i_{n+1}<j$. Notice that   $S_{i_1\ldots i_n}\in\{f_1,\ldots,f_k\}$ by  Lemma \ref{lem: xi(n+1) refines xi(n)}(ii).  If $i_1\ldots i_n<_{\rm lex}\omega(f_i)$, then $\omega(S_{i_1\ldots i_n})\leqslant_{\rm lex} i_1\ldots i_n<_{\rm lex}\omega(f_i)$, which implies that $S_{i_1\ldots i_n}\prec f_i$. Therefore  $S_{i_1\ldots i_n}=f_{\ell}$ and so $h=f_{\ell}\circ S_{i_{n+1}}$ for some $\ell<i$, contradicting the minimality of $i$.  If the other case occurs, i.e. $i_1\ldots i_n=\omega(f_i)$ and $i_{n+1}<j$, then $S_{i_{n+1}}=S_j$. This contradicts our assumption that $S_1,\ldots, S_m\in \Phi$ are distinct.  Hence we have $\omega(h)=\omega(f_i)j$. \qed \end{proof}

\begin{rmk}\label{rmk: concatenation}The conclusion of Lemma \ref{concatenation lem} also holds  if we replace $\Lambda_n(\Delta)$ by $N_n(\Delta)$ and  $\Lambda_{n+1}(\Delta_1)$ by  $N_{n+1}(\Delta_1)$. Indeed,  by the definitions of $N_n(\Delta)$ and $N_{n+1}(\Delta_1)$, it is easily seen that for every $i_1\ldots i_{n+1}\in \Sigma_{n+1}$ with $S_{i_1\ldots i_{n+1}}\in N_{n+1}(\Delta_1)$, we have $S_{i_1\ldots i_n}\in N_n(\Delta)$.  Then the above assertion follows by the same proof of Lemma \ref{concatenation lem}.
\end{rmk}

An essential part to prove Lemma \ref{lem: char vector} is the following result, which says that if  $\phi$ is a similitude which maps $\Lambda_{k}(\Delta_1)$ to $\Lambda_{\ell}(\Delta_2)$ and $N_{k}(\Delta_1)$ to $N_{\ell}(\Delta_2)$ and  preserves the order of the elements of $\Lambda_{k}(\Delta_1)$ and $N_k(\Delta_1)$, then $\phi$ also preserves the order of the elements of $\Lambda_{k+1}(\Delta)$ and $N_{k+1}(\Delta)$ whenever $\Delta$ is an offspring of $\Delta_1$ in $\xi_{k+1}$.

\begin{lem}\label{lemporder}
Let $k, \ell\in \N$, $\Delta_1\in\xi_k$ and  $\Delta_2\in\xi_{\ell}$.  Suppose that   $\Lambda_k(\Delta_1)=(f_1,\ldots,f_p)$, $N_k(\Delta_1)=(g_1,\ldots, g_q)$, $\Lambda_{\ell}(\Delta_2)=(h_1,\ldots, h_{p})$ and $N_{\ell}(\Delta_2)=(u_1,\ldots, u_{q})$, and there is a similitude $\phi$ such that $\phi\circ f_i=h_i$, $\phi\circ g_j=u_j$ for $1\leq i\leq p$ and $1\leq j\leq q$.
Let $\Delta\in\xi_{k+1}$ with $\Delta\subseteq \Delta_1$, and $v,\tilde{v}\in \Lambda_{k+1}(\Delta)$ with $v\neq \tilde{v}$. Then

{\rm (i)} $\phi(\Delta_1)=\Delta_2$.

{\rm (ii)} $v\prec \tilde{v}$ if and only if $ \phi\circ v\prec \phi\circ \tilde{v}$, and the same conclusion  holds if $\Lambda_{k+1}(\Delta)$ is replaced by $N_{k+1}(\Delta)$.
\end{lem}

\begin{proof}
Part (i) of the lemma follows directly from the definition of the elements of $\xi_n(n\geq 0)$ (cf. \eqref{eqDeltaeqleft}) and the assumption that $\phi\circ f_i=h_i$, $\phi\circ g_j=u_j$ for $1\leq i\leq p$ and $1\leq j\leq q$. Below we prove (ii).

Let  $s$ be the smallest integer in $\{1,\ldots, p\}$   so that $v=f_{s}\circ S_{t}$ for some  $t\in \{1,\ldots, m\}$, and  $\tilde{s}$ be the smallest integer in $\{1,\ldots, p\}$  so that $\tilde{v}=f_{\tilde{s}}\circ S_{\tilde{t}}$ for some  $\tilde{t}\in \{1,\ldots, m\}$. Since  $\phi\circ f_i=h_i$ for $1\leq i\leq p$, it is easily seen from the minimality of $s$  that $s$ is also the smallest integer in $\{1,\ldots, p\}$  such that $\phi\circ v=h_{s}\circ S_{t}$.  Similarly, $\tilde{s}$ is the smallest integer in $\{1,\ldots, p\}$ such that $\phi\circ\tilde{v}=h_{\tilde{s}}\circ S_{\tilde{t}}$.  Then by Lemma \ref{concatenation lem}, we have 
\begin{equation}\label{eqomegaveq}
\omega(v)=\omega(f_{s})t, \quad \omega(\phi\circ v)=\omega(h_{s})t, \quad \omega(\tilde{v})=\omega(f_{\tilde{s}})\tilde{t}, \quad \omega(\phi\circ \tilde{v})=\omega(h_{\tilde{s}})\tilde{t}. 
\end{equation}

 Now  by \eqref{eqomegaveq} and  Lemma \ref{lembpoforder},  we have
\begin{align*}\label{eq: comparing order}
\begin{split}
v\prec \tilde{v}&\iff \omega(v)<_{\rm lex}\omega(\tilde{v})\\
&\iff \omega(f_{s})t<_{\rm lex}\omega(f_{\tilde{s}})\tilde{t} \quad  \text{(by \eqref{eqomegaveq})} \\
&\iff \omega(f_{s})<_{\rm lex}\omega(f_{\tilde{s}}), \text{ or } \omega(f_{s})=\omega(f_{\tilde{s}}) \text{ and }t<\tilde{t} \quad \text{(by Lemma \ref{lembpoforder})}\\
&\iff s<\tilde{s}, \text{ or }s=\tilde{s} \text{ and }t<\tilde{t}.
\end{split}
\end{align*}
A similar argument yields that  
\[\phi\circ v\prec \phi\circ\tilde{v}\iff s<\tilde{s}, \  \text{ or }  s=\tilde{s}  \text{ and }  t<\tilde{t}.\]
Therefore, $v\prec \tilde{v}\iff\phi\circ v\prec \phi\circ\tilde{v}$. 
In view of Remark \ref{rmk: concatenation}, it is easy to see that the same conclusion  holds if $\Lambda_{k+1}(\Delta)$ is replaced by  $N_{k+1}(\Delta)$. \qed \end{proof}

\begin{lem}\label{lemDn+1}
Under the assumptions of Lemma \ref{lemporder},   we have $\phi(\Delta)\subseteq \Delta_2$, $\phi(\Delta)\in\xi_{\ell+1}$, $\Lambda_{\ell+1}(\phi(\Delta))=\phi\circ \Lambda_{k+1}(\Delta)$ and $N_{\ell+1}(\phi(\Delta))=\phi\circ N_{k+1}(\Delta)$. \end{lem}

\begin{proof}
By Lemma \ref{lemporder}(i) and  $\Delta\subseteq\Delta_1$, we have  $\phi(\Delta)\subseteq\phi(\Delta_1)=\Delta_2$.   To prove the remaining statements of the lemma, we first show that for every $x\in \Delta_1$, 
\begin{equation}\label{eqScirc}
\phi\circ \Lambda_{k+1}(x)=\Lambda_{\ell+1}(\phi(x)).
\end{equation}
To see this, let $x\in \Delta_1$.  Notice that from Lemma \ref{lem: xi(n+1) refines xi(n)}(ii) we see that 
\begin{equation}\label{eqlambdanp1x}
\Lambda_{k+1}(x)=\left\{f_i\circ S: 1\leq i\leq p, S\in\Phi, x\in f_i\circ S(K)\right\},
\end{equation}
\begin{equation}\label{eqlambdanp1Sx}
\Lambda_{\ell+1}(\phi(x))=\left\{h_i\circ S: 1\leq i\leq p, S\in\Phi, \phi(x)\in h_i\circ S(K)\right\}.
\end{equation}
Since $\phi\circ f_i=h_i$ for $1\leq i\leq p$, \eqref{eqScirc} follows from  \eqref{eqlambdanp1x}-\eqref{eqlambdanp1Sx}.  Since  $\phi\circ g_j=u_j$ for  $1\leq j\leq q$, by \eqref{eqScirc} and a similar argument as above we easily  see that 
\begin{equation}\label{SSIin}
\begin{split}
&\left\{\phi\circ S_I:I\in\Sigma_{k+1}, S_I(K)\cap \left(\bigcap_{f\in\Lambda_{k+1}(x)}f(K)\right)\neq\emptyset\right\}\\
&=\left\{S_J:J\in\Sigma_{\ell+1}, S_J(K)\cap \left(\bigcap_{f'\in\Lambda_{\ell+1}(\phi(x))}f'(K)\right)\neq\emptyset\right\}.
\end{split}
\end{equation}
Now  the lemma  follows easily from \eqref{eqScirc}, \eqref{SSIin} and Lemma \ref{lemporder}. \qed  \end{proof}

We are ready to prove Lemma \ref{lem: char vector}.

\begin{proof}[Proof of Lemma \ref{lem: char vector}]
By Lemma \ref{lemDisch}, it suffices to show that for any $\Delta\in \xi_{k+1}$ with $\Delta\subseteq\Delta_1$, we can find $\Delta'\in \xi_{\ell+1}$ with $\Delta'\subseteq\Delta_2$   such that $\mathcal{C}_{k+1}(\Delta)=\mathcal{C}_{\ell+1}(\Delta')$.  

Write 
$\Lambda_k(\Delta_1)=(f_1,\ldots,f_p)$, $N_k(\Delta_1)=(g_1,\ldots, g_q)$, $\Lambda_{\ell}(\Delta_2)=(h_1,\ldots, h_{p'})$ and $N_{\ell}(\Delta_2)=(u_1,\ldots, u_{q'})$.
Let $\phi=h_1\circ f_1^{-1}$. Then since $\mathcal{C}_{k}(\Delta_1)=\mathcal{C}_{\ell}(\Delta_2)$, we have by definition that  $p=p'$,  $q=q'$,   
\begin{equation}\label{eqmapS}\phi\circ f_i=h_i \quad \text{ and } \quad \phi\circ g_j=u_j \quad \text{ for } 1\leq i\leq p, 1\leq j\leq q.
\end{equation}

Take $\Delta\in\xi_{k+1}$ with $\Delta\subseteq \Delta_1$ and assume that 
\begin{equation}\label{eq:Lambda_n+1 Delta1}
\Lambda_{k+1}(\Delta)=(v_1,\ldots, v_s),\quad N_{k+1}(\Delta)=(w_1,\ldots, w_t).
\end{equation}
Set $\Delta'=\phi(\Delta_1)$. Then by Lemma \ref{lemDn+1},  we have $\Delta'\subseteq \Delta_2$, $\Delta'\in\xi_{\ell+1}$ and 
\begin{equation}\label{eq:claim}
 \Lambda_{\ell+1}(\Delta')=(\phi\circ v_{1}, \ldots, \phi\circ v_{s}), \quad N_{\ell+1}(\Delta')=(\phi\circ w_1,\ldots, \phi\circ w_{t}).  
\end{equation}
It is straightforward to see  from \eqref{eqmapS}-\eqref{eq:claim} that  $\mathcal{C}_{k+1}(\Delta)=\mathcal{C}_{\ell+1}(\Delta')$. Since $\Delta\in\xi_{k+1}$ with $\Delta\subseteq \Delta_1$ is arbitrary,  we complete the proof of  the lemma. \qed
\end{proof}

Recall that $\Omega=\{\mathcal{C}_n(\Delta): \Delta\in \xi_n, n\geq 0\}$ is a finite set,  which in what follows we  view as an alphabet. We say that a word $\alpha_1\alpha_2\ldots\alpha_{\ell}$ over $\Omega$ is {\em admissible} if there exist $k\geq 0$ and $\Delta_i\in \xi_{k+i-1}$ ($i=1,\ldots, \ell$)  such that  
$\Delta_1\supseteq\cdots\supseteq\Delta_{\ell}$
and   $\mathcal{C}_{k+i-1}(\Delta_i)=\alpha_i$ for $1\leq i\leq \ell$. 
 From Lemma \ref{lem: char vector} we see that  the  notion of $\alpha_1\alpha_2\ldots\alpha_{\ell}$ being admissible is independent of the choice of the  sequence $\Delta_1, \Delta_2, \ldots, \Delta_{\ell}$.

Finally, we define the symbolic expression for  each element of $\{\Delta\in \xi_n: n\geq 0\}$. Given $n\geq 0$ and $\Delta\in\xi_n$, let $\Delta_0,\Delta_1,\ldots,\Delta_n$ be the unique sequence of sets satisfying 
\[K=\Delta_0\supseteq \Delta_1\supseteq\cdots\supseteq\Delta_n=\Delta\]
and $\Delta_i\in \xi_i$ for $i=0, 1,\ldots,n$. We call the sequence of characteristic vectors 
\[\mathcal{C}_0(\Delta_0),\mathcal{C}_1(\Delta_1), \ldots, \mathcal{C}_n(\Delta_n)\]
the {\em symbolic expression} of $\Delta$. 

By  Lemma \ref{lemDisch}, we immediately have the following.
\begin{lem}\label{lem: symbolic expression}For any $n\geq 1$ and  $\Delta_1,\Delta_2\in\xi_n $ with $\Delta_1\neq \Delta_2$, the symbolic expressions of $\Delta_1$ and $\Delta_2$ are different. Consequently, for any $n\in \N$ and $\Delta\in\xi_n$, $\Delta$ can be identified as an admissible word of length $n+1$ with initial letter $\mathcal{C}_0(K)$.
\end{lem}

\section{Matrix representation of \texorpdfstring{$\mu$}{mu}: proof of Theorem \ref{main thm:matrix repn}}\label{section: distribution of mu}
In this section we  study the distribution of $\mu$ on each element of $\{\Delta\in\xi_n:n\geq 0\}$. For  $n\geq 0$ and $\Delta\in\xi_n$ with $\mu(\Delta)>0$, we will express $\mu(\Delta)$ as an inner product ${\bf u}_{n, \Delta}\cdot {\bf v}_{n, \Delta}$, where ${\bf u}_{n, \Delta}$, ${\bf v}_{n,\Delta}$ are positive row  vectors with ${\bf v}_{n, \Delta}$ being determined by $\mathcal{C}_n(\Delta)$. If $n\geq 1$, we let $\widehat{\Delta}$ be the unique element of $\xi_{n-1}$ which contains $\Delta$. Then we will construct a transitive matrix $T_{\widehat{\Delta}, \Delta}=T(\mathcal{C}_{n-1}(\widehat{\Delta}), \mathcal{C}_n(\Delta))$ which depends only on  $\mathcal{C}_{n-1}(\widehat{\Delta})$ and $\mathcal{C}_n(\Delta)$ such that ${\bf u}_{n,\Delta}={\bf u}_{n-1,\widehat{\Delta}}T_{\widehat{\Delta},\Delta}$. Below we will give the detailed definitions of ${\bf u}_{n,\Delta}$, ${\bf v}_{n, \Delta}$ and $T_{\widehat{\Delta}, \Delta}$.

For $n\geq 0$, set  
\begin{equation}\label{eqdefFn}
\mathcal{F}_n=\{\Delta\in\xi_n: \mu(\Delta)>0\}.
\end{equation}
 Let $\Delta\in\mathcal{F}_n$ and assume that 
\[\Lambda_n(\Delta)=(f_1,\ldots, f_k),\quad N_n(\Delta)=(g_1,\ldots,g_\ell).\]
Then by definition
\[V_n(\Delta)=(\varphi_1,\ldots, \varphi_k), \quad U_n(\Delta)=(\psi_1,\ldots, \psi_\ell),\]
where  $\varphi_i=f_1^{-1}\circ f_i$, $\psi_j=f_1^{-1}\circ g_j$ for $1\leq i\leq k$ and  $1\leq j\leq \ell$.
Then we define  
\[\Lambda_n^*(\Delta)=(h_1,\ldots, h_{\tilde{k}}), \quad V_n^*(\Delta)=(\phi_1,\ldots, \phi_{\tilde{k}}),\]
where $h_1,\ldots,h_{\tilde{k}}$ (ranked increasingly in the order $\prec$) are those  $f\in\{f_i\}_{i=1}^k$  satisfying $\mu(f^{-1}\Delta)>0$, and $\phi_i:=f_1^{-1}\circ h_i$ for $1\leq i\leq \tilde{k}$. Let $v_n^*(\Delta)$ denote the dimension of $V_n^*(\Delta)$, i.e. $v_n^*(\Delta)=\tilde{k}$. 

We point out that  $V_n^*(\Delta)$ is  determined by  $\mathcal{C}_n(\Delta)$. To see this, first observe that $V_n^*(\Delta)$ is  obtained  by removing those entries $\varphi$ of $V_n(\Delta)$ which satisfy $\mu(\varphi^{-1}\circ f_1^{-1}\Delta)=0$ and keeping the relative positions of the other entries of $V_n(\Delta)$ unchanged. Meanwhile, we note that  $f_1^{-1}(\Delta)$ is determined by $V_n(\Delta)$ and $U_n(\Delta)$, as
\[f^{-1}_1(\Delta)=\left(\bigcap_{i=1}^k\varphi_i(K)\right)\backslash \left(\bigcup_{\psi\in\{\psi_j\}_{j=1}^{\ell}\backslash\{\varphi_i\}_{i=1}^k}\psi(K)\right).\]
Hence $V^*_n(\Delta)$ is determined by $V_n(\Delta)$ and $U_n(\Delta)$, and thus by $\mathcal{C}_n(\Delta)$.

For $I=i_1\ldots i_n\in \Sigma_n$,  write $p_I=p_{i_1}\cdots p_{i_n}$.   Iterating   \eqref{eq: similarity identity} for $n$ times gives 
\begin{equation}\label{eq:self-similar id2}
\mu=\sum_{I\in\Sigma_n}p_I\mu\circ S_I^{-1}.
\end{equation}
Then by \eqref{eq:self-similar id2} and the definition of $\Lambda_n^*(\Delta)$, we have
\begin{align}\label{eq: inner product}
\mu(\Delta)&=\sum_{I\in \Sigma_n}p_I\mu(S_I^{-1}\Delta) \nonumber\\
&=\sum_{I\in\Sigma_n: \  \mu(S_I^{-1}\Delta)>0}p_I\mu(S_I^{-1}\Delta)\nonumber \\
&=\sum_{i=1}^{\tilde{k}}\left(\sum_{I\in\Sigma_n: \ S_I=h_i}p_I\right)\mu(h_i^{-1}\Delta) \nonumber \\
&={\bf u}_{n,\Delta}\cdot {\bf v}_{n,\Delta},
\end{align}
where ${\bf u}_{n,\Delta}=\left(\sum_{I\in \Sigma_n: \ S_I=h_i}p_I\right)_{i=1}^{\tilde{k}}$ and  ${\bf v}_{n,\Delta}=\left(\mu(h_i^{-1}\Delta)\right)_{i=1}^{\tilde{k}}.$ Notice that both ${\bf u}_{n,\Delta}$ and ${\bf v}_{n,\Delta}$ are positive vectors.  Since  $h_i^{-1}(\Delta)=\phi_i^{-1}\circ f_1^{-1}(\Delta)$ for $i=1,\ldots, \tilde{k}$,  by our  argument in the preceding paragraph, we see that ${\bf v}_{n,\Delta}$ is determined  by $\mathcal{C}_n(\Delta)$.

For $n\geq 1$ and $\Delta\in\mathcal{F}_n$, let  $\widehat{\Delta}$ be the unique element of $\mathcal{F}_{n-1}$ which contains $\Delta$. Assume that 
\[\Lambda_{n-1}(\widehat{\Delta})=(u_1,\ldots, u_{k'}), \quad  \Lambda_{n-1}^*(\widehat{\Delta})=(v_1,\ldots, v_{\hat{k}}), \quad V_{n-1}^*(\widehat{\Delta})=(w_1,\ldots, w_{\hat{k}}).\]
Then we define a $v_{n-1}^*(\widehat{\Delta})\times v_{n}^*(\Delta)$ matrix $T_{\widehat{\Delta},\Delta}=(t_{j,i})_{1\leq j\leq \hat{k},\ 1\leq i\leq\tilde{k}}$
by setting
\begin{equation}\label{eq: t_ji}t_{j,i}=\begin{cases}p_r &\text{ if } \exists  r\in\{1,\ldots, m\} \text{ such that } S_r=v_{j}^{-1}\circ h_i,\\
0 &\text{ otherwise},
\end{cases}
\end{equation}
for  $1\leq j\leq \hat{k}$ and $1\leq i\leq\tilde{k}$. 

We claim  that 
\begin{equation}\label{eq: transitive identity}
{\bf u}_{n,\Delta}={\bf u}_{n-1,\widehat{\Delta}}T_{\widehat{\Delta},\Delta}.\end{equation}
To see this, for each $i\in \{1, \ldots, \tilde{k}\}$, let $J\in\Sigma_{n-1}$ and $r\in\{1,\ldots, m\}$ be such that $S_J\circ S_r=h_i$. Notice that $S_J\in\{u_j\}_{j=1}^{k'}$ by  Lemma \ref{lem: xi(n+1) refines xi(n)}(ii). Moreover, by $\Delta\subseteq\widehat{\Delta}$ and  \eqref{eq: similarity identity}, $\mu(S_J^{-1}\widehat{\Delta})\geq\mu(S_J^{-1}\Delta)\geq p_r\mu(S_r^{-1}\circ S_J^{-1}\Delta)=p_r\mu(h_i^{-1}\Delta)>0$. Hence $S_J\in \{v_j\}_{j=1}^{\hat{k}}$ by the definition of $\Lambda^*_{n-1}(\widehat{\Delta})$.
By this fact and \eqref{eq: t_ji},
\begin{align*}\sum_{I\in\Sigma_n: \ S_I=h_i}p_I=\sum_{J\in\Sigma_{n-1}, \ r\in\{1,\ldots, m\}:\  S_J\circ S_r=h_i}p_Jp_r=\sum_{j=1}^{\hat{k}}t_{j, i}\sum_{J\in\Sigma_{n-1}:\  S_J=v_j}p_J.
\end{align*}
Since  $i\in\{1,\ldots, \tilde{k}\}$ is arbitrary, this proves \eqref{eq: transitive identity}. 

Notice that for  $i\in\{1,\ldots,\tilde{k}\}$ and $j\in \{1,\ldots,\hat{k}\}$, we have
\begin{equation}\label{eq29}
 v_j^{-1}\circ h_i=w_j^{-1}\circ u_1^{-1}\circ f_1\circ \phi_i=w_j^{-1}\circ r_n(\Delta)\circ \phi_i.  
\end{equation}
Since we have shown that $V_n^*(\Delta), V_{n-1}^*(\widehat{\Delta})$ are determined by $\mathcal{C}_n(\Delta)$ and $\mathcal{C}_{n-1}(\widehat{\Delta})$ respectively, by \eqref{eq29} and the definition of $t_{j, i}$ (cf. \eqref{eq: t_ji}) we see that $T_{\widehat{\Delta},\Delta}$ is determined by   $\mathcal{C}_n(\Delta)$ and $\mathcal{C}_{n-1}(\widehat{\Delta})$. So we write $T_{\widehat{\Delta},\Delta}=T(\mathcal{C}_{n-1}(\widehat{\Delta}),\mathcal{C}_n(\Delta))$.

Let 
\[\widetilde{\Omega}=\{\mathcal{C}_n(\Delta): \Delta\in\mathcal{F}_n, n\geq 0\},\]
where  $\mathcal{F}_n$ is defined as in \eqref{eqdefFn}. 
For $\alpha\in\widetilde{\Omega}$, pick $n\geq 0$ and $\Delta\in\mathcal{F}_n$ with $\mathcal{C}_n(\Delta)=\alpha$. We have shown that $V_n^*(\Delta)$ and  ${\bf v}_{n,\Delta}$ are determined by $\alpha$ (independent of the choice of $n$ and $\Delta$).  So we can  write  ${\bf v}(\alpha)={\bf v}_{n,\Delta}$ and $v^*(\alpha)=v^*_n(\Delta)$; recall that $v^*_n(\Delta)$ is the dimension of $V_n^*(\Delta)$ and  ${\bf v}_{n,\Delta}$.

For any $\alpha,\beta\in\widetilde{\Omega}$ with  $\alpha\beta$ being  admissible, we have constructed a $v^*(\alpha)\times v^*(\beta)$ dimensional non-negative matrix $T(\alpha,\beta)$. Since $\widetilde{\Omega}$ is finite, we see that
\[\left\{T(\alpha, \beta): \alpha,\beta\in\widetilde{\Omega} \text{ and }\alpha\beta \text{ is admissible}\right\}\]
is a finite family of non-negative matrices.

Recall that ${\bf u}_{0,K}=1$. Applying \eqref{eq: transitive identity} repeatedly and  \eqref{eq: inner product},   we obtain the following.  
\begin{lem}\label{lem: induction}For any $k\geq 1$ and $\Delta\in\mathcal{F}_k$,  we have 
\[{\bf u}_{k,\Delta}=T(\gamma_0,\gamma_1)\cdots T(\gamma_{k-1},\gamma_k),\quad \mu(\Delta)=T(\gamma_0,\gamma_1)\cdots T(\gamma_{k-1},\gamma_k){\bf v}(\gamma_k)^T,
\]
where $\gamma_0\gamma_1\ldots\gamma_k$ is the symbolic expression of $\Delta$ and ${\bf v}(\gamma_k)^T$ is the transpose of ${\bf v}(\gamma_k)$.
\end{lem}

Let $\alpha_1\ldots\alpha_n$ be an admissible word with $\alpha_i\in\widetilde{\Omega}$ for $i=1,\ldots,n$. By the same  proof of \cite[Corollary 3.4]{feng1},
\begin{equation}\label{eq: positive vector}
{\bm e}(\alpha_1)T(\alpha_1,\alpha_2)\cdots T(\alpha_{n-1},\alpha_n)
\end{equation}
is a $v^*(\alpha_n)$-dimensional positive row vector, where ${\bm e}(\alpha_1)$ is the row vector consisting of  $v^*({\alpha_1})$ many $1$'s.

Pick $k\geq 0$ and $\Delta_i\in \xi_{k+i-1}$ ($i=1,\ldots, n$)  such that  
$\Delta_1\supseteq\cdots\supseteq\Delta_{n}$
and   $\mathcal{C}_{k+i-1}(\Delta_i)=\alpha_i$ for $1\leq i\leq n$. Assume that $\Lambda^*_{k}(\Delta_1)=(h_1,\ldots, h_{v^*(\alpha_1)})$ and $\Lambda^*_{k+n-1}(\Delta_n)=\left(h_1',\ldots,h'_{v^*(\alpha_n)}\right)$.  
 By \eqref{eq: t_ji} and induction, we have the following.
\begin{lem}\label{lem: entry expression}For  $i\in\{1,\ldots, v^*(\alpha_1)\}$ and $j\in\{1,\ldots, v^*(\alpha_n)\}$, the $(i,j)$-entry of the matrix $$T(\alpha_1,\alpha_2)\cdots T(\alpha_{n-1},\alpha_n)$$ is given by 
\[\sum_{I\in\Sigma_{n-1}: \ S_I=h_i^{-1}\circ h'_j}p_I.\]
\end{lem}

In the rest of this section, we prove Theorem \ref{main thm:matrix repn} by using Lemma \ref{lem: induction} and a strategy employed  in  the proof of \cite[Lemma 4.1]{feng1}.

Write $\widetilde{\Omega}=\{\eta_1,\ldots, \eta_s\}$ and  let  $N=\sum_{i=1}^sv^*(\eta_i)$. Without loss of generality, assume $\eta_1=\mathcal{C}_0(K)$.  In the following, we construct a family of   $N\times N$ non-negative matrices $\{M_i\}_{i=1}^s$ and a family of  $N$-dimensional positive row vectors $\{\mathbf{w}_i\}_{i=1}^s$. 

For  $i\in \{1,\ldots,s\}$,   we define $M_i$ to  be the partitioned matrix 
\[M_i=(U_{k,j}^i)_{1\leq k, j\leq s},\]
where for $k, j\in\{1,\ldots,s\}$, $U_{k, j}^i$ is a $v^*(\eta_k)\times v^*(\eta_j)$ matrix defined by 
\[U_{k,j}^i=\begin{cases}T(\eta_k,\eta_j) & \text{ if }j=i \text{ and }\eta_k\eta_j \text{ is admissible,}\\
{\bm 0} &\text{ otherwise.}\end{cases}\]
We define a partitioned row vector ${\bf w}_i=(W_j^i)_{1\leq j\leq s}$, where for $j\in\{1,\ldots,s\}$, 
\[W_j^i=\begin{cases}{\bf v}(\eta_j) & \text{ if }j=i,\\
{\bm e}(\eta_j) &\text{ otherwise,}
\end{cases}\]
where ${\bm e}(\eta_j)$ denotes the row vector  consisting of $v^*(\eta_j)$ many $1$'s. It is clear that  ${\bf w}_i$ is an $N$-dimensional positive row  vector. 

\begin{proof}[Proof of Theorem \ref{main thm:matrix repn}]
It follows directly from Lemma \ref{lem: induction}, the  definitions of $\{M_i\}_{i=1}^s$ and $\{\mathbf{w}_i\}_{i=1}^s$, and the product formula of
partitioned matrices. \qed
\end{proof}

\section{Application to \texorpdfstring{$L^q$}{Lq}-spectrum: proof of Theorem \ref{main thm}}\label{section: squared matrices}

To prove Theorem \ref{main thm}, we adopt the same strategy as in \cite{feng1}. Our main target in this section is to show that $\tau(q)=\frac{P(q)}{\log\rho}$ for $q>0$ (cf. Theorem \ref{thm: t(q) and P(q)}), where $P(q)$ is the pressure function for a certain family of squared matrices and $\rho\in (0,1)$ is the common similarity ratio of the similitudes in the IFS. Then Theorem \ref{main thm} follows from a result of Feng and Lau \cite{fenglau1} on the differentiability of $P(q)$ on $(0,\infty)$ under an irreducibility condition.

For this purpose, in Subsection \ref{subs4.1} we construct a family of squared matrices from an essential class $\widehat{\Omega}$ of $\widetilde{\Omega}$ and prove Lemma \ref{lem: uniform vector}. Then we prove in Subsection \ref{subs4.2} that the sum of these matrices is irreducible, which is needed in applying the result of Feng and Lau. Finally, in the last subsection of this section, we prove Theorem \ref{thm: t(q) and P(q)} and completes the proof of Theorem \ref{main thm}. 

\subsection{Essential class \texorpdfstring{$\widehat{\Omega}$}{Omega} of \texorpdfstring{$\widetilde{\Omega}$}{Omega1}}\label{subs4.1}

Recall that  $$\widetilde{\Omega}=\{\mathcal{C}_n(\Delta): \Delta\in\F_n, n\geq 0\}=\{\eta_1,\ldots,\eta_s\},$$ where $\F_n=\{\Delta\in\xi_n: \mu(\Delta)>0\}$. We  call a non-empty subset $\widehat{\Omega}$ of $\widetilde{\Omega}$ an {\em essential class} of $\widetilde{\Omega}$ if $\widehat{\Omega}$  satisfies: (i) $\{\beta\in\widetilde{\Omega}: \alpha\beta\text{ is admissible}\}\subseteq\widehat{\Omega}$ whenever $\alpha\in\widehat{\Omega}$; (ii) for any $\alpha,\beta\in\widehat{\Omega}$, there exists an admissible word $\alpha_1\ldots\alpha_n$ such that $\alpha_1=\alpha$, $\alpha_n=\beta$ and $\alpha_i\in\widehat{\Omega}$ for $1\leq i\leq n$. Such $\widehat{\Omega}$ always exists, see e.g. \cite[Lemma 1.1]{Seneta}.
From now on, we fix an essential class $\widehat{\Omega}$ of $\widetilde{\Omega}$. 

Without loss of generality, we assume that  $\widehat{\Omega}=\{\eta_1,\ldots, \eta_t\}$ for some $1\leq t\leq s$. Let $L=\sum_{i=1}^tv^*(\eta_i)$. Using the same method as in  Section \ref{section: distribution of mu},  we construct a family of  $L\times L$ non-negative matrices $\{M_i\}_{i=1}^t$ and a family of  $L$-dimensional positive row vectors $\{\mathbf{w}_i\}_{i=1}^t$.

Pick $n_0\geq 1$ and $\Delta_0\in\F_{n_0}$ so that $\mathcal{C}_{n_0}(\Delta_0)=\eta_1$. In the following, we will consider the distribution of $\mu$ on the elements of $\F_n$ for $n\geq n_0$ which are contained in $\Delta_0$. 

For $n\geq n_0$ and $\Delta\in\F_n$ with  $\Delta\subseteq\Delta_0$, we define a partitioned vector 
\[\widehat{{\bf u}}_{n,\Delta}=(U_1,\ldots, U_t),\]
where for $i\in\{1,\ldots, t\}$, $U_i$ is a $v^*(\eta_i)$-dimensional row vector defined by
\[U_i=\begin{cases}{\bf u}_{n,\Delta}  & \text{ if }\mathcal{C}_n(\Delta)=\eta_i,\\
{\bm 0} &\text{ otherwise}.
\end{cases}\]
Clearly, $\widehat{{\bf u}}_{n,\Delta}$ is of $L$-dimensional. 

Let $\Theta$ denote the symbolic expression of $\Delta_0$. By Lemma \ref{lem: induction}, the  definitions of $\{M_i\}_{i=1}^t$ and $\{\mathbf{w}_i\}_{i=1}^t$, and the product formula of partitioned matrices, we have
\begin{lem}\label{lem: uniform vector}{\rm (i)} Let $n\geq 0$ and $\Delta\in\F_{n+n_0}$  with $\Delta\subseteq \Delta_0$. Then
\[\mu(\Delta)=\widehat{{\bf u}}_{n_0,\Delta_0}M_{i_1}\cdots M_{i_n}{\bf w}_{i_n}^T,\]
where $\Theta \eta_{i_1}\ldots\eta_{i_n}$ is the symbolic expression of $\Delta$.

{\rm (ii)} A word $\eta_{j_1}\ldots \eta_{j_k}$ over $\widehat{\Omega}$ is admissible if and only if $M_{j_1}\cdots M_{j_k}\neq {\bm 0}$. 
\end{lem}

\subsection{Irreducibility of \texorpdfstring{$\sum_{i=1}^tM_i$}{sum Mi}}\label{subs4.2}

 This subsection is devoted to proving the following result, which is a key step to prove Theorem \ref{main thm}.

\begin{prop}\label{prop: irreducible}Let $H=\sum_{i=1}^t M_i$. Then $H$ is irreducible. That is, there exists a positive  integer $r$ such that all the entries of  $\sum_{i=1}^r H^i$ are positive.
\end{prop}
 
 To prove Proposition \ref{prop: irreducible}, we first give several lemmas. The following result is a consequence of the net structure of $\{\xi_n\}_{n=0}^{\infty}$ and the FTC.
 \begin{lem}\label{lem: finite many mu} $\#\left\{\mu(S_I^{-1}\circ f^{-1}\Delta): I\in\Sigma_*, f\in\Lambda_n(\Delta), \Delta\in\xi_n, n\in\N\right\}<\infty$.
 \end{lem}

\begin{proof} First, fix $n\in\N, \Delta\in\xi_n$ and $f\in\Lambda_n(\Delta)$. Let $k\in \N$ and $I\in\Sigma_k$. By Lemma \ref{lem: xi(n+1) refines xi(n)}, we have 
\begin{equation}\label{eqDD'}
\Delta=\bigcup_{\Delta'\in\xi_{n+k}:\ \Delta'\subseteq\Delta}\Delta',
\end{equation}
 and the sets in the above union are mutually disjoint.  

If $\mu(S_I^{-1}\circ f^{-1}\Delta)>0$, then  we have
\begin{align}\mu(S_I^{-1}\circ f^{-1}\Delta)&=\sum_{\Delta'\in\xi_{n+k}:\ \Delta'\subseteq \Delta}\mu(S_I^{-1}\circ f^{-1}\Delta') \qquad (\text{by }\eqref{eqDD'})\nonumber\\
&=\sum_{\Delta'\in\xi_{n+k}:\ \Delta'\subseteq\Delta, \ \mu(S_I^{-1}\circ f^{-1}\Delta')>0}\mu(S_I^{-1}\circ f^{-1}\Delta') \nonumber\\
&=\sum_{\Delta'\in\F_{n+k}:\ \Delta'\subseteq\Delta,\ f\circ S_I\in\Lambda_{n+k}^*(\Delta')}\mu(S_I^{-1}\circ f^{-1}\Delta'), \label{eq725}
\end{align}
where \eqref{eq725} is due to  the definitions of $\Lambda_{n+k}^*(\Delta')$ for $\Delta'\in\mathcal{F}_{n+k}$. 
By the FTC \eqref{SFTC}, $f\circ S_I(K)$ intersects at most $\#\Gamma$ different elements of $\{S_J(K):J\in\Sigma_{n+k}\}$. Hence there are at most $2^{\#\Gamma}$ many $\Delta'\in \mathcal{F}_{n+k}$ with $f\circ S_I\in \Lambda^*_{n+k}(\Delta')$. This implies that the number of terms in the  sum \eqref{eq725} is at most $2^{\#\Gamma}$. Meanwhile, notice that each term in the sum \eqref{eq725} belongs to the set $\left\{\mu(g^{-1}\Delta'): g\in\Lambda_{\ell}^*(\Delta'), \Delta'\in\F_{\ell},  \ell\in\N\right\}$, which is easily seen to be finite  by \eqref{SFTC} and Lemma \ref{lem: FTC}. Therefore, $\mu(S_I^{-1}\circ f^{-1}\Delta)$ is contained in a finite set independent of $n,  \Delta, f, k$ and $I$.  This completes the proof of the lemma. \qed 
\end{proof}

 The following result plays a key role in the proof of Proposition \ref{prop: irreducible}, whose proof is an application of the Borel density lemma. 

 \begin{lem}\label{lem: mu(E)=1}
 Suppose $E\subset \R^d$ is a Borel set  such that $\mu(S_I^{-1}E)=\mu(E)>0$
 for all $I\in\Sigma_*$. Then $\mu(E)=1$.
 \end{lem}

\begin{proof}For $I\in\Sigma_*$, let $|I|$ be the length of $I$. Write $$[I]=\left\{(x_k)_{k=1}^{\infty}\in \Sigma^{\N}:x_1\ldots x_{|I|}=I\right\}.$$ Let $\eta$ be the infinite Bernoulli product  measure on $\Sigma^{\N}$
generated by the weight $(p_1,\ldots, p_m)$, i.e. $\eta([I])=p_{i_1}\cdots p_{i_n}$ for  $I=i_1\ldots i_n\in\Sigma_*$.  Let $\pi: \Sigma^{\N}\to K$ be the projection map defined by
\[\pi((x_k)_{k=1}^{\infty})=\lim_{n\to\infty}S_{x_1}\circ\cdots\circ S_{x_n}(0) \quad \text{ for }(x_k)_{k=1}^{\infty}\in\Sigma^{\N}.\]
It is well-known that $\mu=\eta\circ \pi^{-1}$.

Set $A=\pi^{-1}(E)$ and $A_I=\pi^{-1}(S_I^{-1}E)$
for $I\in\Sigma_*$. It is clear that $A$ and $A_I$ are Borel subsets of $\Sigma^{\N}$. Moreover, by our assumption we have $\eta(A)=\eta(A_I)>0$.

  We claim that 
\begin{equation}\label{eq: A intersect I}
[I]\cap A=[I]\cap \sigma^{-|I|}(A_I)\quad  \text{ for all } I\in\Sigma_*.
\end{equation}
To see this, notice that for any $I=i_1\ldots i_n\in\Sigma_*$ and $x=(x_k)_{k=1}^{\infty}\in\Sigma^{\N}$,
\begin{align}
x\in A_I&\iff x\in\pi^{-1}(S_I^{-1}E)\nonumber\\
&\iff S_I\circ \pi(x)\in E\nonumber\\
&\iff \pi(Ix)\in E \quad (Ix:=i_1\ldots i_n x_1\ldots)\nonumber\\
&\iff Ix\in A. \label{eqIxA}
\end{align}
Hence,
\begin{align*}
y\in [I]\cap \sigma^{-|I|}(A_I)&\iff y\in [I], \sigma^{|I|}y\in A_I\\
&\iff y\in [I], I\sigma^{|I|}y\in A \quad (\text{by } \eqref{eqIxA})\\
&\iff y\in [I]\cap A,
\end{align*}
from which \eqref{eq: A intersect I} follows.

Next we show that $\eta(A)=1$, which implies that $\mu(E)=1$. Suppose on the contrary that $0<\eta(A)<1$. By the Borel density lemma (see e.g. \cite[Corollary 2.14]{Mattila95}), we have for $\eta$-a.e. $x=(x_k)_{k=1}^{\infty}\in A$,
\[\lim_{n\to\infty}\frac{\eta(A\cap [x_1\ldots x_n])}{\eta([x_1\ldots x_n])}=1.\]
Hence we can find $n\in\N$ and $J\in\Sigma_n$ such that 
\begin{equation}\label{eq: >mu(A)}
\frac{\eta(A\cap [J])}{\eta([J])}>\eta(A).
\end{equation}
However, by \eqref{eq: A intersect I} we have 
\begin{equation}\label{eq: =eta(A)}
\frac{\eta(A\cap [J])}{\eta([J])}=\frac{\eta(\sigma^{-|J|}A_J\cap[J])}{\eta([J])}=\frac{\eta(A_J)\eta([J])}{\eta([J])}=\eta(A),
\end{equation}
where  the second equality is due to   the product property of $\eta$. Thus \eqref{eq: =eta(A)} contradicts  \eqref{eq: >mu(A)}. Hence $\eta(A)=1$ and we are done. \qed
\end{proof}

Lemmas \ref{lem: finite many mu}-\ref{lem: mu(E)=1} have   the following consequence, which is important in the proof of Proposition \ref{prop: irreducible} (indeed Lemma \ref{lem: k row is positive}).

\begin{lem}\label{lem: mu inverse =1}
For any $n\in\N$, $\Delta\in\F_n$ and  $f\in\Lambda_n^*(\Delta)$, there exists $I\in\Sigma_*$ such that $\mu(S_I^{-1}\circ f^{-1}\Delta)=1.$
\end{lem} 

\begin{proof}Fix $n\in\N$, $\Delta\in\F_n$ and $f\in\Lambda^*_n(\Delta)$. By Lemma \ref{lem: finite many mu},  the set $$\left\{\mu(S_J^{-1}\circ f^{-1}\Delta): J\in\Sigma_*\right\}$$
is  finite. So we can find $I\in\Sigma_{*}$ such that 
\begin{equation}\label{eqSfD}
\mu(S_{I}^{-1}\circ f^{-1}\Delta)=\max\left\{\mu(S_J^{-1}\circ f^{-1}\Delta): J\in\Sigma_*\right\}.
\end{equation}
Clearly, $\mu(S_{I}^{-1}\circ f^{-1}\Delta)>0$ as $\mu(f^{-1}\Delta)>0$.

Let $k\in \N$. By  \eqref{eq:self-similar id2} (in which we take $n=k$), we have 
\begin{equation}\label{eq30}
\mu(S_{I}^{-1}\circ f^{-1}\Delta)=\sum_{J\in\Sigma_k}p_J\mu(S_{J}^{-1}\circ S_{I}^{-1}\circ f^{-1}\Delta)=\sum_{J\in\Sigma_k}p_J\mu(S_{IJ}^{-1}\circ f^{-1}\Delta). 
\end{equation}
Then by \eqref{eqSfD}-\eqref{eq30} and the fact that $\sum_{J\in\Sigma_k}p_J=1$, we easily deduce that   
\begin{equation}\label{eq31}
  \mu(S_I^{-1}\circ f^{-1}\Delta)=\mu(S_J^{-1}\circ S_I^{-1}\circ f^{-1}\Delta) \quad \text{ for all }  J\in\Sigma_k.
\end{equation}
Since $k\in \N$ is arbitrary, \eqref{eq31} holds for all $J\in\Sigma_*$. Now it follows from Lemma \ref{lem: mu(E)=1} that $\mu(S_{I}^{-1}\circ f^{-1}\Delta)=1$, completing the proof of the lemma. \qed 
\end{proof}

 The following result follows easily from the definitions of $M_1,\ldots, M_t$, the product formula of  partitioned matrices and induction, whose proof we omit.

\begin{lem}\label{lem: block of product matrices}
Given an admissible word $\eta_{i_1}\ldots \eta_{i_n}$ with $n\geq 2$, write the matrix $M_{i_2}\cdots M_{i_n}$ in the form of partitioned matrix $(U_{i,j})_{1\leq i, j\leq t}$, where $U_{i,j}$ is a $v^*(\eta_i)\times v^*(\eta_j)$ matrix. Then we have $U_{i_1,i_n}=T(\eta_{i_1},\eta_{i_2})\cdots T(\eta_{i_{n-1}},   \eta_{i_n}).$
\end{lem}
 
An essential part  to prove  Proposition \ref{prop: irreducible} is the following result.

 \begin{lem}\label{lem: k row is positive}
 For any $i, j\in\{1,\ldots, t\}$ and $k\in\{1,\ldots, v^*(\eta_i)\}$, there exists an admissible word $\eta_{i_1}\ldots\eta_{i_n}$ with $\eta_{i_1}=\eta_i$ and $\eta_{i_n}=\eta_j$ such that each entry of  the $k$-th row of the matrix $$T(\eta_{i_1},\eta_{i_2})\cdots T(\eta_{i_{n-1}},\eta_{i_n})$$ is  positive.
 \end{lem}
 
 \begin{proof}Recall  that $\widehat{\Omega}=\{\eta_i\}_{i=1}^t$ is an essential class of $\widetilde{\Omega}=\{\mathcal{C}_n(\Delta):\Delta\in\mathcal{F}_n, n\geq 0\}$, and $n_0\in\N$, $\Delta_0\in\mathcal{F}_{n_0}$ are chosen so that $\mathcal{C}_{n_0}(\Delta_0)=\eta_1$. Let $i,j\in\{1,\ldots, t\}$ and $k\in\{1,\ldots, v^*(\eta_i)\}$ be fixed. Pick $n_1\in\N$ and $\Delta_1\in\F_{n_1}$ so that $\Delta_1\subseteq\Delta_0$ and $\mathcal{C}_{n_1}(\Delta_1)=\eta_i$. Assume $\Lambda^*_{n_1}(\Delta_1)=(h_1,\ldots, h_{v^*(\eta_i)})$. By Lemma \ref{lem: mu inverse =1}, there exist $n'\in\N$ and $I\in\Sigma_{n'}$ such that 
 \begin{equation}\label{eq: I,n'}
 \mu(S_{I}^{-1}\circ h_k^{-1}\Delta_1)=1.
 \end{equation}

 Pick $n_2\in\N$ and $\Delta_2\in\F_{n_2}$ with  $\Delta_2\subseteq\Delta_0$ such that 
 \begin{equation}\label{eqmaxv}
 \#\Lambda_{n_2}(\Delta_2)=\max\left\{\#\Lambda_{\ell}(\Delta'):\Delta'\in \mathcal{F}_{\ell}, \Delta'\subseteq\Delta_0, \ell\geq n_0\right\}:=u,
 \end{equation}
 \begin{equation}\label{eqmaxu}
  \#N_{n_2}(\Delta_2)=\max\left\{\#N_{\ell}(\Delta'):\Delta'\in \mathcal{F}_{\ell}, \Delta'\subseteq\Delta_0, \Lambda_{\ell}(\Delta')=u, \ell\geq n_0\right\}:=v.
 \end{equation}
  Assume that $\Lambda_{n_2}(\Delta_2)=(f_1,\ldots, f_{u})$ and $N_{n_2}(\Delta_2)=(g_1,\ldots, g_{v})$. 
Set  $\Delta=h_k\circ S_{I}(\Delta_2)$. Then 
\begin{equation}\label{eq: Delta2}
\Delta=\left(\bigcap_{p=1}^uh_k\circ S_I\circ f_p(K)\right)\backslash\left(\bigcup_{g\in\{g_q\}_{q=1}^v\backslash\{f_{p}\}_{p=1}^u}h_k\circ S_I\circ g(K)\right).
\end{equation}

We assert that $\Delta\in\F_{n'+n_1+n_2}$. To see this, first notice that by the similarity of $\mu$ (see \eqref{eq:self-similar id2}, in which we take $n=n'+n_1$) and \eqref{eq: I,n'}, we have 
 \begin{equation}\label{eq: Delta2 intersect Delta}
 \begin{split}
 \mu(\Delta\cap \Delta_0)&\geq  \mu(\Delta\cap \Delta_1)\\
 &=\mu(h_k\circ S_{I}(\Delta_2)\cap \Delta_1)\\
 &\geq\left(\sum_{J\in\Sigma_{n'+n_1}: \ S_J=h_k\circ S_{I}}p_J\right)\mu(\Delta_2\cap S_{I}^{-1}\circ h_k^{-1}\Delta_1)  \\
 &=\left(\sum_{J\in\Sigma_{n'+n_1}: \ S_J=h_k\circ S_{I}}p_J\right)\mu(\Delta_2)>0 \quad  (\text{by } \eqref{eq: I,n'}).
 \end{split}
 \end{equation}
Since $\mathcal{F}_{n'+n_1+n_2}$ is a partition of $K$ in measure, it follows that  $$\mu(\Delta_0\cap \Delta')\geq \mu(\Delta\cap \Delta_0 \cap \Delta')>0$$ for some $\Delta'\in\mathcal{F}_{n'+n_1+n_2}$. In particular, $\Delta_0\cap \Delta'$ and $\Delta\cap \Delta'$ are both non-empty. Hence  $\Delta'\subseteq \Delta_0$ by the net structure of $\{\xi_n\}_{n\geq 0}$ (cf. Lemma \ref{lem: xi(n+1) refines xi(n)}).  Moreover, by  \eqref{eq: Delta2} and the definition of the mapping $\Lambda_{n'+n_1+n_2}$ on $K$ (cf. \eqref{the map Lambda_n0}), we see that 
\[\Lambda_{n'+n_1+n_2}(x)\supseteq\{h_k\circ S_I\circ f_p\}_{p=1}^u\quad  \text{ for all } x\in \Delta\cap \Delta'.\]
It then follows from the definition of $\Lambda_{n'+n_1+n_2}(\Delta')$ that $$\Lambda_{n'+n_1+n_2}(\Delta')\supseteq\{h_k\circ S_I\circ f_p\}_{p=1}^u.$$
This combining with \eqref{eqmaxv} yields  that indeed
\begin{equation}\label{Lambda(n'+n1+n2)}
\Lambda_{n'+n_1+n_2}(\Delta')=\{h_k\circ S_I\circ f_p\}_{p=1}^u.
\end{equation}
By the definition of $N_{n_2}(\Delta_2)$, 
\[g_q(K)\cap \left(\bigcap_{p=1}^{u}f_p(K)\right)\neq\emptyset, \quad \forall 1\leq q\leq v,\]
which implies that 
\[h_k\circ S_I\circ g_q(K)\cap \left(\bigcap_{p=1}^{u}h_k\circ S_I\circ f_p(K)\right)\neq\emptyset, \quad \forall 1\leq q\leq v.\]
By this fact, the definition of $N_{n'+n_1+n_2}(\Delta')$ and \eqref{eqmaxu}, we see that 
\begin{equation}\label{N(n'+n1+n2)}
N_{n'+n_1+n_2}(\Delta')=\{h_k\circ S_I\circ g_q\}_{q=1}^v.
\end{equation}
Now \eqref{Lambda(n'+n1+n2)}-\eqref{N(n'+n1+n2)} imply that $\Delta=\Delta'$, and so  $\Delta\in\mathcal{F}_{n'+n_1+n_2}$. This proves the above assertion.

 Since $\Delta_1\in\mathcal{F}_{n_1}$, $\Delta\in\mathcal{F}_{n'+n_1+n_2}$ and $\mu(\Delta_1\cap \Delta)>0$ (cf. \eqref{eq: Delta2 intersect Delta}) ,   we have  $\Delta\subseteq \Delta_1$. Assume that $\mathcal{C}_{n'+n_1+n_2}(\Delta)=\eta_{t'}$ for some $t'\in\{1,\ldots, t\}$. Write $\Lambda^*_{n'+n_1+n_2}(\Delta)=\left(h_1',\ldots, h'_{v^*(\eta_{t'})}\right)$.  Then by \eqref{Lambda(n'+n1+n2)} and the fact that $\Delta=\Delta'$, we have
 \begin{equation}\label{eq: Delta_2}
 \left\{h'_1,\ldots, h'_{v^*(\eta_{t'})}\right\}\subseteq\left\{h_k\circ S_{I}\circ f_p\right\}_{p=1}^u.
 \end{equation}
 
 Let $\gamma_0\gamma_1\ldots\gamma_{n_1-1}\eta_i$ be the symbolic expression of $\Delta_1$ and 
 \[\gamma_0\gamma_1\ldots\gamma_{n_1-1}\eta_i\eta_{i_2}\ldots\eta_{i_{n'+n_2}}\eta_{t'}\]
  be that of $\Delta$. By Lemma \ref{lem: entry expression}, for any $1\leq \ell\leq v^*(\eta_{t'})$, the $(k,\ell)$-entry of the matrix 
  \begin{equation}\label{eq8132232}
  T(\eta_i,\eta_{i_2})\cdots T(\eta_{i_{n'+n_2}},\eta_{t'})
  \end{equation}
 is given by 
 \[\sum_{J\in\Sigma_{n'+n_2}:\ S_J=h_k^{-1}\circ h_{\ell}'}p_J,\]
 which is    positive  by \eqref{eq: Delta_2}. Hence each entry of the $k$-th row of \eqref{eq8132232} is  positive.
 
 To finish the proof, pick $n_3\in\N$ and $\Delta_3\in\F_{n'+n_1+n_2+n_3}$ so that $\Delta_3\subseteq \Delta$ with $\mathcal{C}_{n'+n_1+n_2+n_3}(\Delta_3)=\eta_j$. Let $\gamma_0\gamma_1\ldots\gamma_{n_1-1}\eta_i\eta_{i_2}\ldots\eta_{i_{n'+n_2}}\eta_{t'}\eta_{j_1}\ldots\eta_{j_{n_3-1}}\eta_{j}$ be the symbolic expression of $\Delta_3$. Then by the above argument and \eqref{eq: positive vector}, each entry of the $k$-th row of the matrix 
 \[T(\eta_i,\eta_{i_2})\cdots T(\eta_{i_{n'+n_2}}, \eta_{t'})T(\eta_{t'},\eta_{j_1})T(\eta_{j_1},\eta_{j_2})\cdots T(\eta_{j_{n_3-1}}, \eta_{j})\]
 is positive. This completes the proof of the lemma. \qed 
 \end{proof}

 \begin{proof}[Proof of Proposition \ref{prop: irreducible}]
With Lemmas \ref{lem: block of product matrices}-\ref{lem: k row is positive} in hand, the proof of Proposition \ref{prop: irreducible} is identical to that of \cite[Proposition 4.2]{feng1}. We omit the repetition here. \qed
 \end{proof}

\subsection{Proof of Theorem \ref{main thm}}\label{section: proof of main thm}

Let  $M_{1}, \ldots, M_{t}$ be the $L\times L$ matrices  that we have constructed in the beginning of Section \ref{section: squared matrices}. For $q\in\R$, define 
\begin{equation}\label{eqP(q)}
P(q)=\lim_{n\to\infty}\frac{1}{n}\log{\left(\sum{\|M_{i_1}\cdots M_{i_n}\|^q}\right)},
\end{equation}
where   the summation is taken over all words $i_1\ldots i_n\in \{1,\ldots,t\}^n$ with $M_{i_1}\cdots M_{i_n}\neq{\bm 0}$, and $\|A\|=\sum_{i,j}a_{i,j}$ for any non-negative matrix $A=(a_{i,j})_{1\leq i,j\leq N}$. 
The function $P$ is called the {\em pressure function} of $M_{1},\ldots, M_{t}$. Since $\sum_{i=1}^t M_i$ is irreducible (cf. Proposition \ref{prop: irreducible}),   the limit in \eqref{eqP(q)} exists and  $P(q)$ is differentiable on $(0,\infty)$ (see  \cite[Theorem 3.3 and Proposition 4.4]{fenglau1}). This combining with the following result immediately yields Theorem \ref{main thm}, where $\rho\in(0,1)$ is the common similarity ratio of the similitudes in the IFS.  

\begin{thm}\label{thm: t(q) and P(q)}
For $q>0$, $\tau(q)=\frac{P(q)}{\log{\rho}}$.
\end{thm}

The rest of this subsection is devoted to the proof of Theorem \ref{thm: t(q) and P(q)}. We will use the following equivalent definition of $L^q$-spectrum for $q>0$; see  \cite[Proposition 3.1]{LauNgai}. For $q>0$, the $L^q$-spectrum $\tau(q)$ of $\mu$ can be given by
\begin{equation}\label{eq: Lq-spectrum dyadic}
\tau(q)=\liminf_{n\to\infty}\frac{\log{\sum_{D\in\D_n}\mu(D)^q}}{-n\log2},
\end{equation}
where for each $n\in\N$, $\D_n=\left\{\prod_{i=1}^d\left[\frac{k_i}{2^n}, \frac{k_i+1}{2^n}\right): k_i\in\Z \text{ for }1\leq i\leq d\right\}$.

Let $n_0\in\N$ and $\Delta_0\in\F_{n_0}$ be  as in the beginning of  Section \ref{section: squared matrices}. Define $\mu_0=\mu|_{\Delta_0}$, i.e. $\mu_0(A)=\mu(\Delta_0\cap A)$ for any Borel set $A\subset \R^d$.

\begin{prop}\label{prop: calculate tau(q) by Delta}
For $q>0$, we have 
\begin{equation}\label{eqtau(q)mu}
\tau(q)=\liminf_{n\to\infty}\frac{1}{n\log{\rho}}\log{\sum_{\Delta\in\F_n}\mu(\Delta)^q},
\end{equation}
\begin{equation}\label{eqtau(q)mu0}
\tau(\mu_0, q)=\liminf_{n\to\infty}\frac{1}{n\log{\rho}}\log{\sum_{\Delta\in\F_n:\ \Delta\subseteq \Delta_0}\mu(\Delta)^q}.
\end{equation}
\end{prop}

\begin{proof}Fix $q>0$. For each $n\in \N$, let $k_n\in\N$ be such that
\begin{equation}\label{eq:rho kn, 2-n}\rho^{k_n}\leq 2^{-n}<\rho^{k_n-1}.\end{equation}
Then it is clear that there is a constant $N_1$ (independent of $n$) such that each $\Delta\in\xi_{k_n}$ can intersect at most $N_1$ elements of $\D_n$.  On the other hand, since $\Phi$ satisfies the FTC, it is well-known that $\Phi$ satisfies the weak separation condition \cite{Nguyen}. That is, 
 \begin{equation*}\sup_{x\in\R^d, n\in\N}\#\{S_I: S_I(K)\cap B(x,\rho^n)\neq\emptyset, I\in\Sigma_n\}<\infty.
\end{equation*}
By this fact and \eqref{eq:rho kn, 2-n}, it is not hard to  see that there is a constant $N_2$ (independent of $n$) such that each $D\in\D_n$ can intersect at most $N_2$ elements of $\xi_{k_n}$. 

By the above argument,  for any $D\in\D_n$, we have
\begin{equation*}\label{eq:mu(D)^q}\mu(D)^q=\left(\sum_{\Delta\in\xi_{k_n}:\ \Delta\cap D\neq\emptyset}\mu(\Delta)\right)^q\leq N_2^q\sum_{\Delta\in\xi_{k_n}:\ \Delta\cap D\neq\emptyset}\mu(\Delta)^q.\end{equation*}
Hence
\begin{equation*}\label{eq:sum mu(D)^q}\sum_{D\in\D_n}\mu(D)^q\leq N_2^q\sum_{D\in\D_n} \  \sum_{\Delta\in\xi_{k_n}:\ \Delta\cap D\neq\emptyset}\mu(\Delta)^q\leq N_1N_2^q\sum_{\Delta\in\xi_{k_n}}\mu(\Delta)^q.\end{equation*}
Then it follows from   \eqref{eq: Lq-spectrum dyadic}  that
\[\tau(q)\geq\liminf_{n\to\infty}\frac{\log{\sum_{\Delta\in\xi_n}\mu(\Delta)^q}}{n\log{\rho}}=\liminf_{n\to\infty}\frac{\log{\sum_{\Delta\in\F_n}\mu(\Delta)^q}}{n\log{\rho}}.\]
 The `$\leq $' part of \eqref{eqtau(q)mu} can be proved analogously, whose details we  omit. Hence we have proved \eqref{eqtau(q)mu}. To prove \eqref{eqtau(q)mu0}, we simply notice that the above argument still works with slight modifications if we replace $\mu$ by $\mu_0$. \qed 
\end{proof}

\begin{lem}\label{lem: tau(mu0)=tau(mu)}
For $q>0$, $\tau(q)=\tau(\mu_0,q)$.
\end{lem}

\begin{proof}Let $q>0$ be fixed.  It is clear from  Proposition \ref{prop: calculate tau(q) by Delta} that $\tau(q)\leq \tau(\mu_0,q)$. Below we prove that $\tau(q)\geq \tau(\mu_0,q)$.

 By Lemma \ref{lem: mu inverse =1}, there exist $n'\in\N$ and $J\in\Sigma_{n'}$ so that $\mu(S_{J}^{-1}\Delta_0)=1$.
 Let $n\geq 1$ and $\Delta\in\mathcal{F}_n$. Then we have
 \begin{equation}\label{eq: mu_0(SI0(Delta))}
\mu_0(S_{J}(\Delta))^q=\mu(S_{J}(\Delta)\cap\Delta_0)^q\geq p_{J}^q\mu(\Delta\cap  S_{J}^{-1}(\Delta_0))^q=p_{J}^q\mu(\Delta)^q,
\end{equation}
where the second inequality follows from \eqref{eq:self-similar id2} (in which we take $n=n'$) and the last equality holds since $\mu(S_{J}^{-1}\Delta_0)=1$. On the other hand, by Lemma \ref{lem2-3},  $S_{J}(\Delta)$ is a union of some elements of $\xi_{n+n'}$.  Moreover, by the FTC \eqref{SFTC}, it is not hard to see that $S_J(\Delta)$ contains at most $N_0:=2^{\#\Gamma}$ elements of $\mathcal{F}_{n+n'}$.

By the above argument,  for any $n\geq 1$ and $\Delta\in\mathcal{F}_{n}$, 
\[\mu_0(S_{J}(\Delta))^q=\left(\sum_{\Delta'\in\mathcal{F}_{n+n'}:\  \Delta'\subseteq S_{J}(\Delta)}\mu_0(\Delta')\right)^q\leq N_0^q\sum_{\Delta'\in\mathcal{F}_{n+n'}:\ \Delta'\subseteq S_{J}(\Delta)}\mu_0(\Delta')^q.\]
 Therefore,
\begin{equation}\label{eq: sum of mu_0}
\sum_{\Delta\in\mathcal{F}_n}\mu_0(S_{J}(\Delta))^q\leq N_0^q\sum_{\Delta\in\mathcal{F}_n} \ \sum_{\Delta'\in\mathcal{F}_{n+n'}:\ \Delta'\subseteq S_{J}(\Delta)}\mu_0(\Delta')^q\leq N_0^q\sum_{\Delta'\in\mathcal{F}_{n+n'}}\mu_0(\Delta')^q.
\end{equation}
It follows from \eqref{eq: mu_0(SI0(Delta))}-\eqref{eq: sum of mu_0} and Proposition \ref{prop: calculate tau(q) by Delta} that $\tau(q)\geq \tau(\mu_0,q)$, as desired. \qed 
\end{proof}

\begin{proof}[Proof of Theorem \ref{thm: t(q) and P(q)}]
By Lemma \ref{lem: tau(mu0)=tau(mu)}, it suffices to prove that for $q>0$,  $\tau(\mu_0,q)=P(q)/\log{\rho}.$
The proof  is a slight modification of that of \cite[Proposition 5.7]{feng1}. 

 For two vectors ${\bm a}=(a_1,\ldots,a_L)$, ${\bm b}=(b_1,\ldots, b_L)$,  write ${\bm a}\approx{\bm b}$ if there is a constant $C\geq1$ such that $C^{-1}a_i\leq b_i\leq Ca_i$ for all $i$. Let ${\bm e}$ be the row vector consisting of $L$ $1$'s. 
By  the definitions of $\widehat{{\bf u}}_{n_0,\Delta_0}$,  $M_1$ and $\{{\bf w}_i\}_{i=1}^t$, it is direct to see that 
\begin{equation}\label{eq:vnoI}
\widehat{{\bf u}}_{n_0,\Delta_0}\approx {\bm e}M_1 \quad \text{ and } \quad {\bf w}_i\approx {\bm e} \quad \text{ for }  1\leq i\leq t.
\end{equation}

Let $q>0$ be fixed. By Proposition \ref{prop: calculate tau(q) by Delta} and Lemma \ref{lem: uniform vector},  we have
\begin{equation}\label{eq: end}\tau(\mu_0,q)=\liminf_{n\to\infty}\frac{1}{n\log{\rho}}\log\sum{\left(\widehat{{\bf u}}_{n_0,\Delta_0}M_{i_1}\cdots M_{i_n}{\bf w}_{i_n}^T\right)^q},
\end{equation}
where the summation is taken over all admissible words $\eta_1\eta_{i_1}\ldots \eta_{i_n}$. Then  by \eqref{eq:vnoI}-\eqref{eq: end} and Lemma \ref{lem: uniform vector}(ii), we easily see that
\[\tau(\mu_0,q)=\liminf_{n\to\infty}\frac{1}{n\log{\rho}}\log{\sum_{i_1,\ldots, i_n\in\{1,\ldots, t\}:\ M_{1}M_{i_1}\cdots M_{i_n}\neq {\bm 0}}\|M_{1}M_{i_1}\cdots M_{i_n}\|^q}.\]
Now the remaining part of the proof is exactly the same as that of \cite[Proposition 5.7]{feng1}, so we omit the repetition here. \qed 
\end{proof}

\section{The case that \texorpdfstring{$\Phi$}{Phi} is commensurable}\label{sec:comensurable case}

As mentioned in the introduction,  Theorems \ref{main thm:matrix repn}-\ref{main thm} can be extended to  the case that  the IFS is  commensurable and satisfies a more general form of the FTC.

We say that an IFS  $\Phi=\{S_i\}_{i=1}^m$  on $\R^d$ is {\em commensurable
} if there exist $r\in (0,1)$ and positive integers $k_1,\ldots, k_m$ such that   $\rho_i=r^{k_i}$ for $1\leq i\leq m$, where $\rho_i$ is the similarity ratio of $S_i$. Let $K$ be the self-similar set generated by $\Phi$. 

Let $\rho=\min_{1\leq i\leq m}\rho_i$.  Let $\mathcal{A}_0=\{\varepsilon\}$, where $\varepsilon$ is the empty word. For  $n\geq 1$, set
\[\A_n=\left\{i_1\ldots i_{k}\in \Sigma_*: \rho_{i_1}\cdots\rho_{i_k}\leq \rho^n<\rho_{i_1}\cdots\rho_{i_{k-1}}\right\}.\]
We say that $\Phi$ satisfies the finite type condition if there exists a finite set $\Gamma$ such that for any $n\geq 1$ and $I, J\in\A_n$,
\begin{equation}\label{FTC2}
\text{either } \quad  S_I(K)\cap S_J(K)=\emptyset \quad   \text{ or } \quad  S_I^{-1}\circ S_J\in\Gamma.
\end{equation}

We define for each $n\geq 0$, a Borel partition $\xi_n$ of $K$ by using $\mathcal{A}_n$ instead of $\Sigma_n$. Precisely, for  $n\geq 0$, define $\Lambda_n: K\to 2^{\mathcal{S}}$  by 
$$\Lambda_n(x)=\{S_I: I\in \A_n \text{ with }x\in S_I(K)\} \quad  \text{ for }  x\in K.$$
Then  set 
$$\xi_n=\left\{\Lambda_n^{-1}(\mathcal{U}):\mathcal{U}\in\Lambda_n(K)\right\}.$$
It is clear  that $\xi_n$ is a finite Borel partition of $K$. For $\Delta\in \xi_n$, let $\Lambda_n(\Delta)$ be the value of $\Lambda_n$ on $\Delta$. Then define $N_n(\Delta)$ by 
\[N_n(\Delta)=\left\{S_I: I\in\A_n \text{ with } S_I(K)\cap\left(\bigcap_{f\in\Lambda_n(\Delta)}f(K)\right)\neq \emptyset\right\}.\]

Our result in this  section is the following. 

\begin{thm}\label{mainthm:com case}
Let  $\Phi=\{S_i\}_{i=1}^m$ be a commensurable IFS  on $\R^d$  satisfying  \eqref{FTC2}. Let $\mu$ be
the self-similar  measure generated by $\Phi$ and a probability vector $(p_1,\ldots,p_m)$. Then the conclusions of Theorems \ref{main thm:matrix repn}-\ref{main thm} hold.
\end{thm}

The idea to prove Theorem \ref{mainthm:com case}  is essentially the same as that of the equicontractive case but  many details are different. In the following, we point out the major modifications of Sections \ref{section: symbolic expression}-\ref{section: squared matrices}  needed to prove Theorem \ref{mainthm:com case}.

\subsection{The characteristic vectors of \texorpdfstring{$\Delta\in \xi_n$}{Delta in xi n}}
Let the two linear orders $\leqslant_{\rm lex}$  on $\Sigma_*$ and  $\preccurlyeq$ on $\Phi_*$, and the mapping $\omega:\Phi_*\to \Sigma_*$  be the same as in Section \ref{section: symbolic expression}. 

For  $n\geq 0$ and  $\Delta\in\xi_n$, we define the characteristic vector $$\mathcal{C}_n(\Delta)=(V_n(\Delta), U_n(\Delta), r_n(\Delta))$$
 of $\Delta$ in  following way:  Write $\Lambda_n(\Delta)$ and $N_n(\Delta)$ as ordered vectors 
$$\Lambda_n(\Delta)=(f_1,\ldots, f_k),\quad    N_n(\Delta)=(g_1,\ldots, g_{\ell}),$$
where $f_1\prec\cdots\prec f_k$ and $g_1\prec \cdots\prec g_{\ell}$ in $\Phi_*$.  Then  define  $V_n(\Delta)$ and $U_n(\Delta)$ by
$$V_n(\Delta)=((\varphi_1, s_1),\ldots, (\varphi_k, s_k)), \quad  U_n(\Delta)=((\psi_1,t_1), \ldots, (\psi_{\ell}, t_{\ell})),$$
where $\varphi_i=f_1^{-1}\circ f_i$, $s_i=\rho^{-n}\rho_{f_i}$ for $1\leq i\leq k$, and $\psi_j=f_1^{-1}\circ g_j$, $t_j=\rho^{-n}\rho_{g_j}$ for $1\leq j\leq \ell$. Define $r_n(\Delta)$  as in Section \ref{section: symbolic expression}. By \eqref{FTC2} and the assumption that $\Phi$ is  commensurable, it is easily checked that the set $\{\mathcal{C}_n(\Delta): \Delta\in\xi_n,n\geq 0\}$ is finite. 

It is easy to see that all the results in Section \ref{section: symbolic expression} have analogous statements which hold in the setting in this section.

\subsection{The transitive matrices \texorpdfstring{$T_{\widehat{\Delta},\Delta}$}{T(Delta)}} We need some modifications in the definition of the transitive matrices.

For $n\geq 1$, let $\mathcal{F}_n=\{\Delta\in\xi_n:\mu(\Delta)>0\}$. Let $\Delta\in\mathcal{F}_n$ and  assume $\Lambda_n(\Delta)=(f_1,\ldots, f_k)$. Then define
\[\Lambda_n^*(\Delta)=(h_1,\ldots, h_{\tilde{k}}), \quad  V_n^*(\Delta)=((\phi_1, \tau_1),\ldots, (\phi_{\tilde{k}}, \tau_{\tilde{k}})),\]
where $h_1,\ldots, h_{\tilde{k}}$ (ranked increasingly in the order $\prec$) are those  $f\in\{f_i\}_{i=1}^k$ satisfying $\mu(f^{-1}\Delta)>0$, and $\phi_i:=f_1^{-1}\circ h_i$, $\tau_i:=\rho^{-n}\rho_{h_i}$  for $i=1,\ldots,\tilde{k}$. Let $v_n^*(\Delta)$ denote the dimension of $V_n^*(\Delta)$, i.e. $v_n^*(\Delta)=\tilde{k}$. Similar to  Section \ref{section: distribution of mu}, we can show that $V_n^*(\Delta)$ is determined by  $\mathcal{C}_n(\Delta)$.

Let $\widehat{\Delta}\in \mathcal{F}_{n-1}$ be such that  $\Delta\subseteq \widehat{\Delta}$. Assume that $\Lambda_{n-1}^*(\widehat{\Delta})=(h_1',\ldots, h_{\hat{k}}')$. For  $1\leq i\leq \tilde{k}$ and $1\leq j\leq \hat{k}$, we define a number $t_{j,i}$ by 
\[t_{j,i}=\sum p_{W},\]
 where the summation is taken over all words $W\in\Sigma_*$ satisfying that there exists $J\in\A_{n-1}$ such that $JW\in\A_n$, $S_J=h_{j}'$ and $S_{JW}=h_i$. Define $t_{j,i}=0$ if such $W$ does not exist. Then we define a $v_{n-1}^*(\widehat{\Delta})\times v_{n}^*(\Delta)$ 
 matrix by
 \begin{equation*}\label{eq:tran matix}
 T_{\widehat{\Delta},\Delta}=(t_{j,i})_{1\leq j\leq \hat{k},\ 1\leq i\leq\tilde{k}}.
 \end{equation*}
Similar to  Section \ref{section: distribution of mu}, we can see that $T_{\widehat{\Delta}, \Delta}$ is determined by $\mathcal{C}_{n-1}(\widehat{\Delta})$ and $\mathcal{C}_n(\Delta)$. 

Using the transitive matrices constructed above,  all the results in Sections \ref{section: distribution of mu}-\ref{section: squared matrices} can be extended to the setting in this section, without using new ideas. As a consequence, Theorem \ref{mainthm:com case} is proved. We omit the details. 

{\noindent \bf  Acknowledgements}.
The author is grateful to his supervisor, De-Jun Feng, for many helpful discussions, suggestions and in particular pointing out Lemma \ref{lem: mu(E)=1}.  He also wish to thank  the anonymous referee for his/her suggestions that led to the improvement of the paper.

\end{document}